\documentclass[12pt]{amsart}

\usepackage{amssymb}
\usepackage{amsmath,amscd,stmaryrd,amsthm,amsrefs} 
\usepackage{mathptmx}  

\title{Dyadic Steenrod algebra and its applications}

\author{A. S. Janfada}
\address{Department of Mathematics, Urmia University, Urmia 57561-51818, Iran}
\email{asjanfada@gmail.com; a.sjanfada@urmia.ac.ir}

\author{Gh. Soleymanpour}
\address{Department of Mathematics, Urmia University, Urmia 57561-51818, Iran}
\email{g.soleymanpour@yahoo.com; gh.soleymanpour@urmia.ac.ir}


\newtheorem{theorem}{Theorem}[section]

\newtheorem{corollary}[theorem]{Corollary}
\newtheorem{proposition}[theorem]{Proposition}

\theoremstyle{definition}
\newtheorem{definition}[theorem]{Definition}
\newtheorem{example}[theorem]{Example}
\newtheorem{problem}[theorem]{Problem}

\theoremstyle{remark}
\newtheorem{remark}[theorem]{Remark}

\newcommand{\rank}{\rm{rank}}
\newcommand{\ppt}{{\rm Part}}

\begin{document}

\begin{abstract}
First, by inspiration of the results of Wood \cite{differential,problems}, but with the methods of non-commutative geometry and different approach, we extend the coefficients of the Steenrod squaring operations from the filed $\mathbb{F}_2$ to the dyadic integers $\mathbb{Z}_2$ and call the resulted operations the dyadic Steenrod squares, denoted by $Jq^k$. The derivation-like operations $Jq^k$ generate a graded algebra, called the dyadic Steenrod algebra, denoted by $\mathcal{J}_2$ acting on the polynomials $\mathbb{Z}_2[\xi_1, \dots, \xi_n]$. Being $\mathcal{J}_2$ an Ore domain, enable us to localize $\mathcal{J}_2$ which leads to the appearance of the integration-like operations $Jq^{-k}$ satisfying the $Jq^{-k}Jq^k=1=Jq^kJq^{-k}$. These operations are enough to exhibit a kind of differential equation, the dyadic Steenrod ordinary differential equation. Then we prove that the completion of $\mathbb{Z}_2[\xi_1, \dots, \xi_n]$ in the linear transformation norm coincides with a certain Tate algebra. Therefore, the rigid analytic geometry is closely related to the dyadic Steenrod algebra. Finally, we define the Adem norm $\| \ \|_A$ in which the completion of $\mathbb{Z}_2[\xi_1, \dots, \xi_n]$ is $\mathbb{Z}_2\llbracket\xi_1,\dots,\xi_n\rrbracket$, the $n$-variable formal power series. We surprisingly prove that an element $f \in \mathbb{Z}_2\llbracket\xi_1,\dots,\xi_n\rrbracket$ is hit if and only if $\|f\|_A<1$. This suggests new techniques for the traditional Peterson hit problem in finding the bases for the cohit modules.
\end{abstract}

\maketitle
\tableofcontents

\section{Introduction}\label{intro}

The structure of this article is entirely algebraic but its idea comes from topology. 
Unlike its name, the Steenrod algebra is born of topology and has arisen from the heart of cohomology. This algebra was presented for the first time in the studies of Norman Earl Steenrod of mod 2 cohomology operations \cite{steenrod} and extended by himself and others \cite{cartan2,milnor,silverman,singer2,problems}. On the power and ability of the Steenrod algebra, it is enough to say that this algebra behaves like the integers. Just as each abelian group is a module over the integers, the cohomology ring of each topological space is a module over the Steenrod algebra  \cite{mosher}. This strength of the Steenrod algebra caused the algebraist to define clear definitions without any complexity  like Eilenberg-MacLane complexes, acyclic carriers, cup-$i$ products and etc. Some comprehensive references of such attempts are \cite{may,mosher,smith2}.

The squaring operations $Sq^k$, the generators of the Steenrod algebra, treat, somehow, like the derivation operator. The problem in abusing the derivation property is that, since these operations involve the misbehavior field $\mathbb{F}_2=\{0,1\}$, the power of maneuver gets out of hand. Even the Peterson hit problem, which could be solved by a system of linear equations, turns to a complicated open problem because of misbehaving of $\mathbb{F}_2$. In \cite{differential}, by derivation-like essence of the squaring operations, the coefficients extend to $\mathbb{Z}$ but, this generalization was not comprehensive. Even extending the coefficients to the finite field $\mathbb{F}_q, q=2^\nu$, $\nu \ge 2$ in \cite{smith2} could not progressed more. 

This paper pursues two aims. First, the paper aims to consider the derivation-like property of the squaring operations in analytic point of view which  appears in rigid analytic geometry, a field of mathematics that is the fourth gate to start and pass through the algebraic geometry, after the algebraic, analytic, and categorical gates.

The second aim is to exhibit new techniques for the traditional Peterson hit problem \cite{criterion,kameko,mothebe,peterson,sum} so that to find a basis for the cohit modules, we are not compelled to examine each monomial individually. 

To reach both aims, there is a need for the Steenrod squaring operations $Sq^k$ to be escaped from the restriction of the  misbehavior field $\mathbb{F}_2$ to have more freedom provided that the nature of the $Sq^k$ is not deranged. To this end, by inspiration of the the results of Wood \cite{differential,problems}, but with the methods of non-commutative geometry and different approach, the mod 2 Steenrod algebra is generalized to any ring $R$ with identity for which there exists a nonzero homomorphism $\phi: R \to \mathbb{F}_2$. The especial case $R={\mathbb Z}$ is the one that Wood called the integral Steenrod algebra. It can be identified with the Landweber-Novikov algebra of cobordism theory \cite{buch-ray,miller}. Passing from ${\mathbb Z}$ into ${\mathbb Z}_2$, we answer to Problems 2.22 and 2.23 posed in \cite{problems} in the literature of the dyadic Steenrod algebra.

Everything starts when we take $R=\mathbb{Z}_2$, the dyadic integers, and a homomorphism $\phi: \mathbb{Z}_2 \to \mathbb{F}_2$ and then induce it to the homomorphism 
\[ \phi :  \mathbb{Z}_2[\xi_1, \dots, \xi_n] \to \mathbb{F}_2[\xi_1, \dots, \xi_n]. \]
This motivates us to define the dyadic Steenrod operations 
\[ Jq^k: \mathbb{Z}_2[\xi_1, \dots, \xi_n]^d \to  \mathbb{Z}_2[\xi_1, \dots, \xi_n]^{d+k}, \]
for $k \ge0$ and $d \ge0$, such that the following diagram commutes.
\[
\begin{CD}
\text{$ \mathbb{Z}_2[\xi_1, \dots, \xi_n]^d$} @ > Jq^k >> \text{$ \mathbb{Z}_2[\xi_1, \dots, \xi_n]^{d+k}$} \\
@V \text{$\phi$} VV  @VV \text{$\phi$}V \\ 
\text{$ \mathbb{F}_2[\xi_1, \dots, \xi_n]^d$} @ >\text{$Sq^k$}>> \text{$\mathbb{F}_2[\xi_1, \dots, \xi_n]^{d+k}$}
\end{CD} 
\]
The base of our work originates from this definition in Section \ref{genso}. In this section we study some properties of the $Jq^k$. In particular, the $Jq^k$ satisfy the dyadic version of Cartan formula. Then we show that the $Jq^k$, as linear transformations on $\mathbb{Z}_2[\xi_1,\dots,\xi_n]$, are bounded. In a future paper we will extend the coefficients to any ring $R$ with identity.

In Section \ref{diadsa},  we present the dyadic Steenrod algebra $\mathcal{J}_2$ as a graded (non-commutative) algebra generated by the dyadic squaring operations $Jq^k$ subject to $Jq^0=1$ and the dyadic Adem expansions, which will be establish. We show that the $\mathbb{Z}_2$-algebra $\mathcal{J}_2$ is generated by $\{Jq^{2^n}\}_{n \ge 0}$. We even more extend the coefficients ring of $\mathcal{J}_2$ to $\mathbb{Q}_2$ and prove that $\{Jq^1,Jq^2\}$ is the minimal generating set for $\mathcal{J}_2$ as a $\mathbb{Q}_2$-algebra. The Hopf algebra structure of $\mathcal{J}_2$ and the canonical conjugation in $\mathcal{J}_2$ are considered in this section.

In Section \ref{divrf}, we prove that $\mathcal{J}_2$ is an Ore domain and hence the left classical ring of fractions of $\mathcal{J}_2$ exists. Then the operations $Jq^{-k}$ are appeared from the localization of the ring $\mathcal{J}_2$. The operations $Jq^k$ and $Jq^{-k}$ play roles of, respectively, the derivation and the integration operators in the usual calculus. They satisfy the $Jq^kJq^{-k}=1=Jq^{-k}Jq^k$, the former equation is interpreted as if the derivation of the integration of a function is the function itself.  It should be noted that, unlike the usual calculus, here the derivation-like operation $Jq^k$ increases the degree and the integration-like operation $Jq^{-k}$ decreases it. These studies led to the appearance of a kind of differential equation, the dyadic Steenrod ordinary differential equation which we denote by {$\Sigma$}ODE. The next task is localizing the polynomials, as a module over $\mathcal{J}_2$, which we will show the elements of the resulted localized algebra is in one-to-one correspondence with the solutions of the {$\Sigma$}ODE's with constant coefficients.

In Section \ref{lint}, inspiring by the linearity of the $Jq^k$, we define the linear transformation norm $\| \ \|_L$ and get the interesting result $B_L(0,1)=\ker(\phi: \mathcal{J}_2 \to \mathcal{A}_2)$, where $B_L(0,1)$ is the unit ball in this norm. As a consequence of this fact, we show that the topology induced by the norm $\| \ \|_L$ on $\mathcal{J}_2$ is the $\ker(\phi)$-adic topology, and exhibit a completion $(\widehat{\mathcal{J}_2})_L$  of $\mathcal{J}_2$ and study the nature of its elements. Then we show that the completion of the polynomial algebra $\mathbb{Z}_2[\xi_1,\dots,\xi_n]$, as a $(\mathcal{J}_2, \| \ \|_L)$-module, is nothing but $T_n(\mathbb{Z}_2)$, the Tate algebra over $\mathbb{Z}_2$. With all these in mind, a rigid analytic geometry starts with polynomials and the dyadic Steenrod algebra. Therefor, we reveal that the rigid analytic geometry is closely related to the dyadic Steenrod algebra. In fact, the Tate algebras are, in one hand, the base of the rigid analytic geometry and, on the other hand, the completion of polynomials as a module over the dyadic Steenrod algebra $(\mathcal{J}_2, \| \ \|_L)$.

In Section \ref{ademn}, we define the multiplicative Adem norm $\| \ \|_A$ and construct another completion  $(\widehat{\mathcal{J}_2})_A$ of $\mathcal{J}_2$. We study the elements of $(\widehat{\mathcal{J}_2})_A$ and show that among them there is a generalization of the dyadic Steenrod total operation $Jq$ and the elements satisfy the generalized Cartan formula. We also construct another completion of the polynomial algebra $\mathbb{Z}_2[\xi_1,\dots,\xi_n]$ as a module over $(\mathcal{J}_2, \| \ \|_A)$. This completion is $\mathbb{Z}_2\llbracket\xi_1,\dots,\xi_n\rrbracket$, the $n$-variable formal power series whose elements are, surprisingly, concerned with the hit problem: The element $f \in \mathbb{Z}_2\llbracket\xi_1,\dots,\xi_n\rrbracket$ is hit if and only if $\|f\|_A<1$. In other words the set of hit elements is 
\[ H(n)=B_A(0,1)\mathbb{Z}_2 \llbracket \xi_1,\dots,\xi_n \rrbracket, \]
where $B_A(0,1)$ is the unit ball in the Adem norm.  Therefor, efficient techniques are obtained to find the generators of $Q(n)$, instead of examining the elements individually. We determine the elements of $Q(1)$ over $\mathbb{Z}_2\llbracket\xi\rrbracket$.

The final section is allotted to some comments (with hints) on the future of the problem and state some open problems.


\section{Generalization of the Steenrod operations}\label{genso}
Cohomology operations are algebraic operations on the cohomology groups of spaces which commute with the homomorphisms induced by continuous mappings. They are used to decide questions about the existence of continuous mappings which cannot be settled by examining cohomology groups alone \cite{steenrod}.

Let $p$ and $q$ be fixed positive integers and $G$ and $G'$ fixed $R$-modules, where $R$ is any ring. A cohomology operation $\theta$ of type $(p,G;q,G')$ is a natural transformation from the cohomology functor $H^p(-;G)$ with coefficients in $G$ to the cohomology functor $H^q(-;G')$ with coefficients in $G'$.  Thus $\theta$ assigns to a topological space $X$ a function 
\[ \theta_X: H^p(X;G) \rightarrow H^q(X;G'), \]
such that if $f:X \rightarrow Y$ is a continuous map, there is a commutative square 
\[\begin{CD}
\text{$H^p(Y;G)$} @ >\text{$\theta_Y$}>> \text{$H^q(Y;G')$} \\
@V \text{$f^*$} VV  @VV \text{$f^*$}V \\ 
\text{$H^p(X;G)$} @ >\text{$\theta_X$}>> \text{$H^q(X;G')$}
\end{CD} 
\]

\noindent Hinging on this, we generalize the mod $2$ Steenrod algebra $\mathcal{A}_2$. Let $R$ be a ring with identity and consider the graded $R$-algebra $R[\xi_1, \dots, \xi_n]$ of all polynomials in $n$ variables $\xi_1,\dots,\xi_n$ with coefficients in $R$ graded by $\deg(\xi_i)=1$ for $1 \le i \le n$.  Then 
\[ R[\xi_1, \dots, \xi_n]=\sum_{d \ge 0}  R[\xi_1, \dots, \xi_n]^d, \]
where $ R[\xi_1, \dots, \xi_n]^d$ is the (left) $R$-module of polynomials of degree $d$ in $ R[\xi_1, \dots, \xi_n]$. We identify both cases $n=0$ and $d=0$ with $R$. The algebra $ R[\xi_1, \dots, \xi_n]$ is commutative only if $R$ is so. 

The algebra $ R[\xi_1, \dots, \xi_n]$ is freely generated by the $\xi_i$, as well as the monomials $\xi_1^{d_1}\xi_2^{d_2} \dots \xi_n^{d_n}$ with $d_1+d_2+ \cdots +d_n=d$ and $d_i \ge 0$ for $1 \le i \le n$, form a basis for $ R[\xi_1, \dots, \xi_n]^d$. If all exponents $d_i=0$ this is the identity of $ R[\xi_1, \dots, \xi_n]^0=R$. 

We start our generalizing with a ring $R$ with identity for which a nonzero ring homomorphism $\phi: R \to \mathbb{F}_2$ exists. This homomorphism induces a homomorphism
\[ \phi :  R[\xi_1, \dots, \xi_n] \to \mathbb{F}_2[\xi_1, \dots, \xi_n]. \]
Now assume $Y$ is a space that $H^*(Y;R)= R[\xi_1, \dots, \xi_n]$. Then we have the following diagram for any $k \ge0$ and $d \ge0$:
\[
\begin{CD}\label{d}
\text{$ R[\xi_1, \dots, \xi_n]^d$} @ > ? >> \text{$ R[\xi_1, \dots, \xi_n]^{d+k}$} \\
@V \text{$\phi$} VV  @VV \text{$\phi$}V \\ 
\text{$ \mathbb{F}_2[\xi_1, \dots, \xi_n]^d$} @ >\text{$Sq^k$}>> \text{$\mathbb{F}_2[\xi_1, \dots, \xi_n]^{d+k}$}
\end{CD} 
\]
\begin{remark}
The existence of the spaces $Y$ above does not affect our work. In non-commutative geometry, also, we face the same problem where, we talk of manifolds whose ring of differentiable functions is non-commutative while such manifolds do not exist.
\end{remark}

We need an $R$-operation $Sq_R^k$: $ R[\xi_1, \dots, \xi_n]^d \rightarrow  R[\xi_1, \dots, \xi_n]^{d+k}$ for which the above diagram commutes. That is, for any polynomial $f \in  R[\xi_1, \dots, \xi_n]^d$, 
\[ \phi(Sq_R^k(f))=Sq^k(\phi(f)). \]

\begin{example}\label{1}
Let $R=\mathbb{F}_4=\{0,1,\alpha,\alpha+1\}$, where $\alpha$ is a primitive cubic root of unity. Define the homomorphism $\phi:\mathbb{F}_4 \to \mathbb{F}_2$ by
\begin{equation}\label{phi}
 \phi(0)=0, \phi(1)=1, \phi(\alpha)=1, \phi(\alpha+1)=0. 
\end{equation}
We want to determine $Sq_{\mathbb{F}_4}^1$: ${\mathbb{F}_4}[\xi]^1 \rightarrow {\mathbb{F}_4}[\xi]^2$ such that $\phi(Sq_{\mathbb{F}_4}^1(f))=Sq^1(\phi(f))$ for any $f \in \mathbb{F}_4[\xi]^1$. Here, $\phi : \mathbb{F}_4[\xi] \to \mathbb{F}_2[\xi]$. By a straightforward calculation using \eqref{phi} we get
\[ \phi(Sq_{\mathbb{F}_4}^1(1+\alpha \xi+ (\alpha+1)\xi^2))=\xi^2. \]
So, $Sq_{\mathbb{F}_4}^1(1+\alpha \xi+ (\alpha+1)\xi^2) \in \phi^{-1}(\xi^2)$ for which there are many choices like:
\[ \xi^2, \alpha \xi^2, \alpha+1+\xi^2, \alpha+1+\alpha \xi^2, \dots. \]
\end{example}
As seen in Example \ref{1}, the $R$-operations $Sq_R^k$ are not well-defined. To improve this, we strengthen our main definition.
\begin{definition}\label{sqrk}
The total Steenrod $R$-square 
\[ Sq_R:  R[\xi_1, \dots, \xi_n] \to  R[\xi_1, \dots, \xi_n] \]
is defined as an $R$-algebra homomorphism by
\[ Sq_R(1)=1, \ Sq_R(\xi_i)=\xi_i+\xi_i^2, \ \text{for} \ 1 \le i \le n. \] 
The Steenrod $R$-square 
\[Sq_R^k:  R[\xi_1, \dots, \xi_n]^d \to  R[\xi_1, \dots, \xi_n]^{d+k}, \]
for $k \ge 0$ and $d \ge 0$, is the $R$-module homomorphism defined by restricting $Sq_R$ to $ R[\xi_1, \dots, \xi_n]^d$ and projecting on $ R[\xi_1, \dots, \xi_n]^{d+k}$. Thus, 
\[ Sq_R=\sum_{k \ge 0}Sq_R^k, \]
the formal sum of its graded parts. The Steenrod $R$-squares $Sq^k_R$ satisfy the commutative property
\begin{equation}\label{comm}
\phi(Sq_R^k(f))=Sq^k(\phi(f)). 
\end{equation}
\end{definition} 
\begin{example}\label{1-1}
In Example \ref{1}, for $f=1+\alpha \xi+ (\alpha+1)\xi^2$, we have
\begin{align*}
Sq_{\mathbb{F}_4}(1) &= 1, \\
Sq_{\mathbb{F}_4}(\alpha \xi) &= \alpha \xi + \alpha \xi^2,\\
Sq_{\mathbb{F}_4}((\alpha + 1)\xi^2) &= (\alpha + 1)\xi^2 + 2(\alpha + 1)\xi^3 + (\alpha + 1)\xi^4.
\end{align*}
It follows that
\begin{align*}
Sq_{\mathbb{F}_4}^0(f) &= f,\\
Sq_{\mathbb{F}_4}^1(f) &=\alpha \xi^2 + 2(\alpha + 1)\xi^3=\alpha \xi^2,\\
Sq_{\mathbb{F}_4}^2(f) &=(\alpha+1) \xi^4,\\
Sq_{\mathbb{F}_4}^k(f) &=0 \ \text{for} \ k \ge 3.
\end{align*}
\end{example}

Some special cases of the Steenrod $R$-squares $Sq_R^k$ are well known. In the case $R=\mathbb{F}_2$ the $Sq_R^k$ are the so-called mod $2$ Steenrod squares $Sq^k$. If we take $R=\mathbb{Z}$ we get the integral Steenrod squares $SQ^k$ \cite{differential}. For $R=\mathbb{F}_q, q=2^\nu$, $\nu \ge 2$, we have the Steenrod reduced power operations ${\mathcal P}^k$ \cite{smith2}.
 
We are interested in the case $R=\mathbb{Z}_2$, the dyadic integer numbers. For simplicity, we shall adopt the following notations:
\begin{equation}\label{notation}
 Jq:=Sq_{\mathbb{Z}_2} \  \text{and for $k \ge 0$,} \ Jq^k:=Sq^k_{\mathbb{Z}_2}, 
\end{equation}
and call $Jq^k$ (resp. $Jq$) as dyadic (resp. total) Steenrod square.
\begin{proposition}\label{z2linear}
For $k \ge 0$, the dyadic Steenrod square $Jq^k$ is $\mathbb{Z}_2$-linear.
\end{proposition}
\begin{proof}
The result is clear by the linearity of $Jq$.
\end{proof}

With the notations in \eqref{notation} we rewrite Definition \ref{sqrk}.
\begin{definition}\label{totaljq}
The dyadic total Steenrod square 
\[ Jq:  \mathbb{Z}_2[\xi_1, \dots, \xi_n] \to  \mathbb{Z}_2[\xi_1, \dots, \xi_n] \]
is a $\mathbb{Z}_2$-algebra homomorphism defined by
\[ Jq(1)=1, \ Jq(\xi_i)=\xi_i+\xi_i^2, \ \text{for} \ 1 \le i \le n. \] 

For $k \ge 0$ and $d \ge 0$, the dyadic Steenrod square 
\[Jq^k:  \mathbb{Z}_2[\xi_1, \dots, \xi_n]^d \to  \mathbb{Z}_2[\xi_1, \dots, \xi_n]^{d+k}, \]
is the restriction of $Jq$ to $ \mathbb{Z}_2[\xi_1, \dots, \xi_n]^d$. The formal sum 
\[ Jq=\sum_{k \ge 0}Jq^k\] 
is now clear.

Finally, by a dyadic Steenrod operation we mean a linear transformation 
\[ \theta:  \mathbb{Z}_2[\xi_1, \dots, \xi_n] \rightarrow  \mathbb{Z}_2[\xi_1, \dots, \xi_n] \]
obtained by the addition and multiplication of the $Jq^k$. 
\end{definition} 
The next result follows immediately.
\begin{proposition}\label{jqc} 
$Jq^k(a)=0$ for all $k>0$ and all $a \in \mathbb{Z}_2$.
\end{proposition}

The multiplication property of $Jq$ gives the dyadic version of Cartan formula.
\begin{proposition}[Cartan formula]
For polynomials $f,g \in  \mathbb{Z}_2[\xi_1, \dots, \xi_n]$ and $k \ge 0$,
\[ Jq^k(fg)=\sum_{i+j=k} Jq^i(f) Jq^j(g). \]
\end{proposition}

We give some further properties of the dyadic Steenrod squares in the next four results.  The proofs are clear from Definition \ref{totaljq}. 
Evaluation in single variables is considered in the next result. 
\begin{proposition}
For all $1 \le i \le n$, 
\[ Jq^k(\xi_i^d)=\binom{d}{k} \xi_i^{d+k}, \]
where the binomial coefficients are taken in $\mathbb{Z}_2$.
\end{proposition}
The next result shows how to evaluate a dyadic Steenrod square on a monomial.
\begin{proposition}
Let $f=\xi_1^{d_1}\xi_2^{d_1}\cdots \xi_n^{d_n}$ be a monomial in $ \mathbb{Z}_2[\xi_1, \dots, \xi_n]$. Then, for all $k>0$,
\[ Jq^k(f)=\sum_{k_1+k_2+\cdots k_n=k} Jq^{k_1}(\xi_1^{d_1}) Jq^{k_2}(\xi_2^{d_2}) \cdots Jq^{k_n}(\xi_n^{d_n}). \]
\end{proposition}
As the reason of naming the Steenrod squaring operation, the next result shows why $Jq^k$ is called squaring operation.
\begin{proposition}
For any monic monomial $f \in  \mathbb{Z}_2[\xi_1, \dots, \xi_n]^d$, 
\[ Jq^k(f)=
\begin{cases}
f^2, \text{ if $k=d$}\\ 
0, \text{ if $k>d$.}
\end{cases} \]
\end{proposition}
The dyadic Steenrod squares does not exceed the variables.
\begin{proposition}
For all polynomials $f \in  \mathbb{Z}_2[\xi_1, \dots, \xi_n]$ and all $k \ge 0$, every monomial in $Jq^k(f)$ involves exactly the same variables as $f$ does.
\end{proposition}

The ring of coefficients of the $Jq^k$ may also been extended to $ \mathbb{Q}_2$, the field of dyadic numbers. 
\begin{theorem}
The operations $Jq^k$, for $k \ge 0$, are $\mathbb{Q}_2$-linear.
\end{theorem}
\begin{proof}
It suffices to prove $Jq^k(qf)=qJq^k(f)$ for all $q \in \mathbb{Q}_2$ and $f \in  \mathbb{Z}_2[\xi_1, \dots, \xi_n]$. Each member of $\mathbb{Q}_2$ is of the form 
\[ q=\sum_{m=m_0}^\infty a_m2^m, \ m_0 \in \mathbb{Z}, \ a_m=0,1. \]
We claim that $Jq^1(q)=0$. To prove the claim, let first $m$ be a positive integer. Then, $2^m \in \mathbb{Z}_2$ and $Jq^1(2^m)=0$. On the other hand, by Cartan formula, we have 
\[ 0=Jq^1(1)=Jq^1(2^m \frac{1}{2^m}) = 2^m Jq^1(\frac{1}{2^m})+Jq^1(2^m) \frac{1}{2^m} = 2^m Jq^1(\frac{1}{2^m}). \] 
Hence $Jq^1(\frac{1}{2^m})=0$. We showed that $Jq^1(2^m)=0$ for all $m \in \mathbb{Z}$. This proves the claim. Now, by induction and using Cartan formula, we get $Jq^k(q)=0$, for all $k>0$, completing the proof.
\end{proof}

In the studies of polynomial algebras over $R[\xi_1, \dots,\xi_n]$, with $R$ a non-Archimedean ring, the non-Archimedean norm is defined to be the maximum coefficient over the coefficients ring. For more details see \cite{bosch}. Take any finite subset $N$ of $\mathbb{Z}_+^n$. For simplicity, for any multi-index $J=(j_1,j_2,\dots, j_n) \in N$, we shall adopt the following notations.
\begin{equation}\label{mi}
a_J:=a_{j_1 j_2 \dots j_n}, \ \xi^J:=\xi_1^{j_1}\xi_2^{j_2}\cdots \xi_n^{j_n}.
\end{equation}
\begin{definition}\label{nan}
For the element 
\[ f=\sum_{J \in N} a_J \xi^J \in  \mathbb{Q}_2[\xi_1, \dots, \xi_n], \]
the non-Archimedean norm of $f$ is defined by
\[ \| f \|_2 =  \max_{J \in N} \{ |a_J |_2 \}, \]
where $| \ |_2$ stands for the dyadic absolute value on $\mathbb{Q}_2$. 
\end{definition}
The next result provides the boundedness property of the $Jq^k$ in the non-Archimedean norm.

\begin{theorem}\label{compare}
For any element $f=\sum_{J \in N} a_J \xi^J \in  \mathbb{Q}_2[\xi_1, \dots, \xi_n]$ and any $k>0$ we have $\| Jq^k(f) \|_2 \le \|f \|_2$.
\end{theorem}
\begin{proof}
We first prove the result for monomials. For $k>0$, $Jq^k$ acts on the monomial $\xi^J=\xi_1^{j_1}\xi_2^{j_2}\cdots \xi_n^{j_n}$ as follows.
\begin{align*}
Jq^k(\xi^J) &= \xi_1^{j_1} Jq^k(\xi_2^{j_2}\cdots \xi_n^{j_n}) \\
& \quad + \binom{j_1}{1} \xi_1^{j_1+1} Jq^{k-1}(\xi_2^{j_2}\cdots \xi_n^{j_n}) \\
& \quad + \dots  \\
& \quad + \binom{j_1}{k} \xi_1^{j_1+k} \xi_2^{j_2}\cdots \xi_n^{j_n}.
\end{align*}

All coefficients in the right hand side are integers with dyadic absolute value $\le 1$. Therefore,
\begin{equation}\label{mono}
\| Jq^k(\xi^J) \|_2 \le 1= \| \xi^J \|_2. 
\end{equation}
Now we take the general element $f=\sum_{J \in N} a_J \xi^J \in  \mathbb{Q}_2[\xi_1, \dots, \xi_n]$. Then, using the equation \eqref{mono} we have
\begin{align*}
\| Jq^k(f) \|_2 & = \| \sum_{J \in N} a_J Jq^k(\xi^J) \|_2 \\
& \le \max_{J \in N} \{ |a_J |_2 \} \\
& = \|f \|_2.
\end{align*}
\end{proof}
The following corollary is now clear.
\begin{corollary}\label{bound}
The operations $Jq^k, k \ge 0$ are bounded in the non-Archimedean norm. 
\end{corollary}
\begin{definition}\label{ltn}
The linear transformation norm $\| \ \|_L$ for $\theta \in \mathcal{J}_2$ is defined by 
\[ \| \theta \|_L = \inf \{ c \ge 0 \mid \forall f \in \mathbb{Q}_2[\xi_1, \dots, \xi_n], \|\theta(f)\|_2 \le c\|f\|_2 \}. \]
\end{definition}
We shall study the linear transformation norm in details in Section \ref{lint}. The following corollary is immediately follows from Theorem \ref{compare} and the fact that $Jq^k(\xi^k)=\xi^{2k}$ for any single variable $\xi$. 
\begin{corollary}\label{boundcor}
For any $k \ge 0$, $\| Jq^k \|_L=1$.
\end{corollary}
\section{Dyadic Steenrod algebra}\label{diadsa}
The dyadic Steenrod algebra, denoted by $\mathcal{J}_2$, is defined as the non-commutative $\mathbb{Z}_2$-algebra generated by the $Jq^k, k \ge 0$, subject to the relation $Jq^0=1$, the identity element of $\mathcal{J}_2$. That is,
\[ \mathcal{J}_2= \mathbb{Z}_2 \langle Jq^k | k \ge 0 \rangle. \]
The homomorphism  $\phi:\mathbb{Z}_2 \to \mathbb{F}_2$ together with the commutative property \eqref{comm} in Definition \ref{sqrk} induces the ring homomorphism $\phi: \mathcal{J}_2 \rightarrow \mathcal{A}_2$ given by $\phi(Jq^k)=Sq^k$. Boundedness of the operations $Jq^k$ enable us to think $\mathcal{J}_2$ as a normed algebra. However, $\mathcal{J}_2$ is not a Banach algebra since the series $\sum_{k=0}^\infty 2^k(Jq^1)^k$ is Cauchy but not convergent. In fact, for any single variable $\xi$,
\[ \sum_{k=0}^\infty 2^k(Jq^1)^k(\xi)=\sum_{k=0}^\infty 2^k k! \xi^{k+1} \notin  \mathbb{Z}_2[\xi_1, \dots, \xi_n]. \]
We shall exhibit a completion of $\mathcal{J}_2$ in the linear transformation norm in Section \ref{lint} .

Now we settle the dyadic version of Adem relations. For $a \le 2b$, the Adem relations
\begin{equation}\label{ad}
R(a,b)=Sq^aSq^b-\sum_{j=0}^{[a/2]} \binom{b-j-1}{a-2j} Sq^{a+b-j}Sq^j,
\end{equation}
where $[a/2]$ denotes the greatest integer $\le a/2$, are the most important and basic relations in the Steenrod algebra $\mathcal{A}_2$. Using the Adem relations \eqref{ad}, all members of $\mathcal{A}_2$ can be represented in terms of admissible monomials. This is the Adem relations \eqref{ad} that enable us to define $\deg(Sq^k)=k$ and introduce $\mathcal{A}_2$ as a graded algebra. For more details see \cite{steenrod}.

Unfortunately, the same relations as \eqref{ad} does not exist in the dyadic Steenrod algebra $\mathcal{J}_2$. However, in the sequel we show that there are combinations of the elements in each homogeneous degree with coefficients not so convenient. This is enough to give $Jq^k$ degree $k$, for $k>0$, and make $\mathcal{J}_2$ be a graded algebra. Thus, the nonzero monomial $Jq^K$ with exponent vector $K=(k_1,k_2,\dots,k_s)$ has degree $|K|=k_1+k_2+\cdots+k_s$.

For $k \ge 0$, we denote by $\mathcal{J}_2^k$ the (left) submodule over $\mathbb{Z}_2$ spanned by the set of monomials $Jq^K$ of degree $k$. We write $\deg(\theta)=k$ if $\theta \in \mathcal{J}_2^k$. Therefor $\mathcal{J}_2=\sum_{k \ge 0} \mathcal{J}_2^k$. We also denote by $\mathcal{J}_2^+=\sum_{k > 0} \mathcal{J}_2^k$, the two-sided ideal of $\mathcal{J}_2$ generated by $\{Jq^k\}_{k>0}$.

Our next aim is to  study the combinations which leads to the dyadic Adem expansions.

Degree one is straightforward. There is only one operation $Jq^1$ and $\rank(\mathcal{J}_2^1)=1$. In degree two we have $Jq^2, Jq^1Jq^1$ which are independent. For, if $aJq^2+bJq^1Jq^1=0$ for some $a,b \in \mathbb{Z}_2$, then effecting both sides to any single variable $\xi$ we have
\[ 0=aJq^2(\xi)+bJq^1Jq^1(\xi)=bJq^1(\xi^2)=2b\xi^3. \]
Then $b=0$ and hence $a=0$. Therefor, $\rank(\mathcal{J}_2^2)=2$. 

In degree three we concern with the operations 
\[ Jq^3, \ Jq^1Jq^2, \ Jq^2Jq^1, \ Jq^1Jq^1Jq^1. \]
 By a straightforward calculations we see that $Jq^1Jq^2$, $Jq^2Jq^1$, $Jq^1Jq^1Jq^1$ are independent. However, there is a nontrivial combination of $Jq^3$ in terms of $Jq^1Jq^2$, $Jq^2Jq^1$, $Jq^1Jq^1Jq^1$.
\begin{proposition} \label{adem3}
There exists $a,b,c,d \in \mathbb{Z}_2$,  not all zero, which 
\[ A_3=aJq^3+bJq^2Jq^1+cJq^1Jq^2+dJq^1Jq^1Jq^1=0. \]
\end{proposition}
\begin{proof}
For any $m>0$, effecting the above equation to the $m$-th power $\xi^m$ of any single variable $\xi$ gives the following polynomial of degree $3$ in $m$.
\[ (3a+3b+3c+6d)m^3+(-3a+3b+3c+18d)m^2+(2a-6c+12d)m=0. \]
Since $m$ is arbitrary, the coefficients must be zero. Solving the resulted system of three linear equations of four unknowns we get
\begin{equation}\label{a3}
A_3=3Jq^3-6Jq^2Jq^1+3Jq^1Jq^2+Jq^1Jq^1Jq^1=0, 
\end{equation}
with $A_3(\xi^m)=0$. 

By linearity, for any element $\zeta=\sum_{i=0}^n a_i\xi^i$, in $\mathbb{Z}_2[\xi]$ we have $A_3(\zeta)=0$. Now for any pair of one variable polynomials $\eta=\sum_{i=0}^n a_i\xi_u^i$ and $\gamma=\sum_{i=0}^l b_i\xi_v^i$, in terms of distinct variables $\xi_u,\xi_v \in \{\xi_1, \dots, \xi_n\}$, we obtain $A_3(\eta \gamma)=0$. Here, $\eta \gamma \in \mathbb{Z}_2[\xi_u,\xi_v]$. It follows, by induction, that $A_3(\delta)=0$, for all $\delta \in  \mathbb{Z}_2[\xi_1, \dots, \xi_n]$ and the proof is completed.
\end{proof} 
Compare the relation \eqref{a3} with Example 2.12 of \cite{problems}.

Proposition \ref{adem3} shows that $\rank(\mathcal{J}_2^3)=3$. We call $A_3$  the dyadic Adem expansion of $Jq^3$ and extend it for higher degrees.
\begin{theorem}[Dyadic Adem expansion]\label{ademk}
For any $k \ge 4$, there are integers $a_1$, $a_2,\dots,a_k$, not all zero, such that
\begin{equation}\label{ak}
A_k= a_kJq^k+a_{k-1}Jq^{k-1}Jq^1+ \cdots + a_1Jq^1Jq^{k-1}=0. 
\end{equation} 
\end{theorem}
\begin{proof}
For any $m>0$, effecting both sides of the  equation \eqref{ak} to $\xi^m$, for any single variable $\xi$, gives a polynomial of degree $k$ in $m$.
\begin{align*}
A_k(\xi^m) &= \sum_{i=0}^{k-1} a_{k-i}Jq^{k-i}Jq^i(\xi^m) \\
           &= \sum_{i=0}^{k-1}a_{k-i} \binom{m}{i} \binom{m+i}{k-i} \\
					 &= m(m-1) (b_{k-2}m^{k-2}+b_{k-3}m^{k-3}+ \cdots +b_0), 
\end{align*}
where each of the $k-1$ coefficients $b_j, 0 \le j \le k-2,$ is a linear combinations of the $k$ unknowns $a_1, a_2, \dots, a_k$.

Now, equating the coefficients $b_j$ to zero, gives a system of at most $k-1$ equation in $k$ unknowns which has always nontrivial solutions. By a similar argument as in the last part of the proof of Proposition \ref{adem3} the result follows.
\end{proof}
Theorem \ref{ademk} answers to Problem 2.23 in \cite{problems} in the language of the dyadic Steenrod algebra.

Through this article, all partitions of an integer assumed to be ordered.
\begin{definition}
The expansion  
\[ A_k= a_kJq^k+a_{k-1}Jq^{k-1}Jq^1+ \cdots + a_1Jq^1Jq^{k-1}, \]
which involves the $2$-partitions of $k$, is called the dyadic Adem expansions of $Jq^k$. 
\end{definition}
In Theorem \ref{ademk} and its proof we concern with $2$-partitions. There are also other expansions for $Jq^k$ in terms of $t$-partitions, $2 \le t \le k-1$. The only strict $k$-partition has $k$ copies of $1$ and can not expand $Jq^k$. Note that by a $t$-partition expansion we mean an expansion captured by $t$-partitions. Nevertheless, a few other partitions may be occurred.  For example the following is a $3$-partition expansion of $Jq^4$.
\[ 24Jq^4-12Jq^1Jq^1Jq^2+12Jq^1Jq^2Jq^1-12Jq^2Jq^1Jq^1+Jq^1Jq^1Jq^1Jq^1=0. \]
The above discussion says that, for $k \ge 4$, we have many options in choosing the coefficients in the expansions by various partitions. 
\begin{example}\label{a345} 
In the following we give the dyadic Adem expansions of $A_k$ for $k=4,5,6$.
\begin{gather*} 
A_4 = 2Jq^4-3Jq^3Jq^1+Jq^2Jq^2+Jq^1Jq^3=0, \\
A_5 = 5Jq^5-5Jq^4Jq^1+Jq^2Jq^3-2Jq^1Jq^4=0, \\
A_6 = 9Jq^6-7Jq^5Jq^1+Jq^2Jq^4+3Jq^1Jq^5=0.
\end{gather*}
\end{example}
As seen in the proof of Theorem \ref{ademk}, for $k \ge 4$, the factor $m(m-1)$ appears in all terms of the polynomial 
\[ A_k(\xi^m) =\sum_{i=0}^{k-1}a_{k-i} \binom{m}{i} \binom{m+i}{k-i}, \ a_i \in \mathbb{Z}_2. \]

For numerical reasons, for $k=3t+i, t \ge 2, i=1,2,3$, the factor $m(m-1) \cdots (m-t)$ appears in all terms of $A_k(\xi^m)$. Thus, from the expansions of $A_{3t+1}, A_{3t+2}, A_{3t+3}$ we obtain a systems of $k-t$ linear equations in $k$ unknowns. Hence, all but $t$ unknowns are found in terms of $t$ independent ones. 

\begin{remark}\label{options}
The above argument, for $k \ge 7$, says that we have infinitely many options to choose the coefficients for almost any purpose. These options are in addition to those quoted before Example \ref{a345} for $k \ge 4$. In particular, always, we may take $a_k \ne 0$.
\end{remark}
\begin{example}\label{ex}
Write $A_7(a,b)$ for the dyadic Adem expansion $A_7$ corresponding to the independent unknowns $a,b \in \mathbb{Z}_2$. By calculation, 
\begin{align}
A_7(a,b) &= \bigl( \frac{-14}{3}a+\frac{14}{3}b \bigr)Jq^7+\bigl( \frac{29}{3}a-\frac{14}{3}b \bigr)Jq^6Jq^1 \label{e3} \\
         & \quad +\bigl( \frac{-28}{3}a+\frac{7}{3}b \bigr)Jq^5Jq^2+\bigl( \frac{28}{15}a-\frac{7}{15}b \bigr)Jq^4Jq^3 \notag \\
				 & \quad+\bigl( \frac{4}{3}a-\frac{1}{3}b \bigr)Jq^3Jq^4+aJq^2Jq^5+bJq^1Jq^6=0. \notag 
\end{align}
For example, for $a=0$, 
\begin{align*}
A_7(0,1) &= \frac{14}{3}Jq^7-\frac{14}{3}Jq^6Jq^1+\frac{7}{3}Jq^5Jq^2 \\
         & \quad -\frac{7}{15}Jq^4Jq^3-\frac{1}{3}Jq^3Jq^4+Jq^1Jq^6=0. 
\end{align*}
\end{example}
 
As an interesting fact, only two elements span the $\mathbb{Q}_2$-algebra $\mathcal{J}_2$.
\begin{theorem}\label{gen}
As a $\mathbb{Q}_2$-algebra, $\mathcal{J}_2$ is generated by $Jq^1$ and $Jq^2$, that is, 
\[ {\mathcal{J}_2}\otimes_{\mathbb{Z}_2} \mathbb{Q}_2 = \mathbb{Q}_2 \langle Jq^1,Jq^2 \rangle. \]
\end{theorem}
\begin{proof}
Clearly $Jq^1, Jq^2$  are linearly independent and we have already seen that $aJq^2+bJq^1Jq^1=0$, for $a,b \in \mathbb{Z}_2$, if and only if $a=b=0$.  On the other hand, the relation \eqref{a3} in the proof of Proposition \ref{adem3} shows that 
\begin{equation}\label{jq3}
Jq^3=2Jq^1Jq^2-Jq^2Jq^1-\frac{1}{3}Jq^1Jq^1Jq^1.
\end{equation}
Note that $\frac{1}{3} \in \mathbb{Z}_2 \subset \mathbb{Q}_2$. For $k \ge 4$ we consider the Adem expansion
\[ A_k= a_kJq^k+a_{k-1}Jq^{k-1}Jq^1+ \cdots + a_1Jq^1Jq^{k-1}=0. \] 
By Remark \ref{options}, we may take $a_k \ne 0$. By an induction now the result follows.
\end{proof}
Compare Theorem \ref{gen} with Theorem 3.20 of \cite{differential}.
\begin{example}
Combining the Adem expansion of $A_4$ in Example \ref{a345} with the relation \eqref{jq3} we have
\begin{multline*}
Jq^4=3Jq^1Jq^2Jq^1-\frac{3}{2}Jq^2Jq^1Jq^1-Jq^1Jq^1Jq^2 \\
+\frac{1}{2}Jq^1Jq^2Jq^1-\frac{1}{2}Jq^2Jq^2-\frac{1}{3}Jq^1Jq^1Jq^1Jq^1. 
\end{multline*}
This shows that $Jq^4$ is not $\mathbb{Z}_2$-generated by $Jq^1$ and $Jq^2$ since not all coefficients are in $\mathbb{Z}_2$. 

In $\mathbb{Z}_2$ we have two types of elements: even and non-even. Even numbers  are not invertible in $\mathbb{Z}_2$ since the dyadic absolute values of them are at most $\frac{1}{2}$. In fact, the inverse of an even number is in $\mathbb{Q}_2 \setminus \mathbb{Z}_2$. By a non-even number we mean a number with dyadic absolute value $1$ which are invertible in $\mathbb{Z}_2$. Odd numbers, of course, belong to this type but do not cover them. For example $\frac{1}{3}, \frac{5}{21}$ are also non-even.
\end{example}
In the expansion of $A_7$  in Example \ref{ex} it seems that $Jq^7$ is also not $\mathbb{Z}_2$-indecomposable. However, as the next result shows, the only $\mathbb{Z}_2$-indecomposables are $Jq^{2^n}$, $n \ge 0$.
\begin{theorem}\label{jq2n}
The $\mathbb{Z}_2$-algebra $\mathcal{J}_2$ is generated by $\{Jq^{2^n}\}_{n \ge 0}$.
\end{theorem}
\begin{proof}
We first show that for any $n \ge 0$, $Jq^{2^n}$ is indecomposable in $\mathcal{J}_2$, as $\mathbb{Z}_2$-algebra. Let 
\begin{align}
A_{2^n} &=\sum_{i=0}^{2^n-1} a_{2^n-i} Jq^{2^n-i} Jq^i \notag \\ 
        &=  a_{2^n} Jq^{2^n}+\sum_{i=1}^{2^n-1} a_{2^n-i} Jq^{2^n-i} Jq^i=0 \label{a2n}
\end{align} 
be the dyadic Adem expansion of $Jq^{2^n}$. Repeatedly dividing the coefficients in the equation \eqref{a2n}  by 2, if necessary, we see that at least one of the coefficients must be non-even. If $a_{2^n}$ is even then since we have  non-even coefficients in dyadic Adem expansion of $Jq^{2^n}$, there are terms with coefficients not in $\mathbb{Z}_2$. This shows that $Jq^{2^n}$ is $\mathbb{Z}_2$-indecomposable. We show  that $a_{2^n}$ can not be non-even. Otherwise, we claim that all coefficients $a_i$ with $1 \le i < 2^n$ must be even. To prove the claim, let $a_{j_1}, a_{j_2}, \dots, a_{j_l},$ with $1 \le j_i < 2^n$ for $1 \le i \le l$, are  non-even. Then, using the homomorphism $\phi: \mathcal{J}_2 \to \mathcal{A}_2$, we get
\[ 0=\phi(0)=\phi(A_{2^n})=Sq^{2^n}+ \sum_{i=1}^l Sq^{j_i}Sq^{k-j_i}, \]
which is impossible since $Sq^{2^n}$ is indecomposable in $\mathcal{A}_2$. This proves the claim. Now effecting $A_{2^n}$ to $\xi^{2^n}$, for any single variable $\xi$, and canceling $\xi^{2^{n+1}}$, we get
\[ a_{2^n}+\sum_{i=1}^{2^n-1} a_{2^n-i} \binom{2^n}{i} \binom{2^n+i}{2^n-i}=0, \]
which is true only if $a_{2^n}$ is even, contradicts the selection of $a_{2^n}$. Similarly, the coefficient of $Jq^{2^n}$ in each expansion involving a $t$-partition, $3 \le t <2^n$,  must be even. Therefor, $Jq^{2^n}$  are indecomposable. 

Now we show that if $k$ is not a power of $2$, then $Jq^k$ is decomposable in terms of some $Jq^{2^i}$ with $2^i<k$. Let ${\rm BPart}(k)$ be the set of  binary partitions of $k$. Then by a similar method as used for the dyadic Adem expansion involving $2$-partitions, for $k \ge 3$, there is a nontrivial equation 
\[ a_kJq^k+ \sum_{I \in {\rm BPart}(k)} a_IJq^I =0. \]
We may take $a_k \ne 0$. Furthermore, we can rearrange the coefficients in such a way that $a_k$ is non-even. This can always be done. For, the decomposition of the form 
\[Sq^k=\sum_{J \in {\rm BPart}(k)} Sq^J, \]
exists since $\{ Sq^{2^n}\}$ is a basis for $\mathcal{A}_2$. Hence for some (any) non-even numbers $\alpha, \beta \in  \mathbb{Z}_2$,
\[ \alpha Jq^k-\beta \sum_{J \in {\rm BPart}(k)} Jq^J \in \ker(\phi).\]
Thus, for some $\theta \in \mathcal{J}_2^k$ with $\phi(\theta)=0$ we have
\[ \alpha Jq^k - \beta \sum_{J \in {\rm BPart}(k)} Jq^J + \theta =0. \]
Now the result follows by induction on degree of $\theta$.
\end{proof}
Theorem \ref{jq2n} answers to Problem 2.22 of \cite{problems} in dyadic Steenrod algebra point of view.
\begin{example}
For $k=7$ we have
\begin{multline*}
210Jq^7+\frac{280}{3} Jq^4Jq^1Jq^2+60 Jq^1Jq^2Jq^4-\frac{700}{9} Jq^1Jq^4Jq^2\\
-15 Jq^2Jq^1Jq^4+125 Jq^2Jq^4Jq^1+14Jq^4Jq^1Jq^1Jq^1=0. 
\end{multline*}
We also have the following expansion in which the coefficient of $Jq^7$ is non-even.
\begin{multline*}
15Jq^7+6Jq^4Jq^2Jq^1+\frac{31}{3} Jq^4Jq^1Jq^2+\frac{65}{7} Jq^1Jq^2Jq^4-\frac{145}{9} Jq^1Jq^4Jq^2\\
-\frac{60}{7} Jq^2Jq^1Jq^4+\frac{170}{21} Jq^2Jq^4Jq^1+Jq^4Jq^1Jq^1Jq^1=0.
\end{multline*}
\end{example} 

We now study the Hopf algebra structure of $\mathcal{J}_2$. As the Steenrod algebra $\mathcal{A}_2$, the diagonal map $\psi: \mathcal{J}_2 \to \mathcal{J}_2 \otimes \mathcal{J}_2$ defined by
\[ \psi(Jq^k)=\sum_{i+j=k} Jq^i \otimes Jq^j \]
as well as the augmented homomorphism $\varepsilon:\mathcal{J}_2 \to \mathbb{Z}_2$ given by
\[ \varepsilon(Jq^k)=
\begin{cases}
1, & \text{if $k=0$} \\
0, & \text{if $k \ne 0$} 
\end{cases} \]
make the dyadic Steenrod algebra $\mathcal{J}_2$ to be a Hopf algebra.
 
Each Hopf algebra has a canonical conjugation $\chi$ defined as follows. Given a Hopf algebra $H$, for any $f,g: H \to H$ the convolution product $f*g$ is the composition
\[ H \xrightarrow{\psi} H \otimes H \xrightarrow{f \otimes g} H \otimes H \xrightarrow{\mu} H, \] 
where $\mu$ is the algebra multiplication. This composition makes ${\rm Hom}(H,H)$ to be a group. The inverse of ${\rm id}: H \to H$ in this product is called the conjugation map $\chi$, that is, ${\rm id}*\chi=\chi*{\rm id}=1$. For more details see \cite{milnor}.

The next result explains the conjugation in $\mathcal{J}_2$.
\begin{proposition}[Thom formula]\label{thom}
The (canonical) conjugation $\chi: \mathcal{J}_2 \to \mathcal{J}_2$ is given by the following recursive formulas.
\[ \chi(Jq^0)=1 \text{ and for $k>0$, } \sum_{i+j=k} Jq^i \chi(Jq^j)=0. \]
\end{proposition}
For example, $\chi(Jq^1)=-Jq^1$, $\chi(Jq^2)=Jq^1Jq^1-Jq^2$. It is clear that $\phi(\chi(Jq^k))=\chi(Sq^k)$. In the next result we  exhibit an explicit formula for handling with $\chi(Jq^k)$. Following \cite{milnor2}, for $k>0$, we denote by $\ppt(k)$ the (ordered) partitions of $k$. 
\begin{theorem}
If $\chi: \mathcal{J}_2 \to \mathcal{J}_2$ is the canonical conjugation then, for $k>0$,
\[ \chi(Jq^k)=\sum_{\alpha \in \ppt(k)} (-1)^{{\rm len}(\alpha)} Jq^\alpha, \]
where ${{\rm len}(\alpha)}$ is the length of $\alpha$.
\end{theorem}
\begin{proof}
The proof is by induction on $k$. For $k=1$ the result is clear as $\chi(Jq^1)=-Jq^1$. Assume the result for the values less than $k$. Then we have
\[ \chi(Jq^k)+Jq^1\chi(Jq^{k-1})+Jq^2\chi(Jq^{k-2})+\cdots+Jq^k=0. \]
By the hypothesis of induction we have
\begin{align*}
\chi(Jq^k) &= -Jq^1 \sum_{\alpha \in \ppt(k-1)} (-1)^{{\rm len}(\alpha)} Jq^\alpha \\
           & \quad -Jq^2 \sum_{\alpha \in \ppt(k-2)} (-1)^{{\rm len}(\alpha)} Jq^\alpha - \cdots - Jq^k. 
\end{align*}					
The result now follows by the fact that  $\ppt(k)=\{(k)\} \bigcup_{\alpha \in \ppt(k-i)} \{(i,\alpha)\}$.  
\end{proof} 
As a consequence in the usual Steenrod algebra, we have the following corollary using $\phi: \mathcal{J}_2 \to \mathcal{A}_2$.
\begin{corollary}
For any $k>0$ we have $\chi(Sq^k)=\sum_{\alpha \in \ppt(k)} Sq^\alpha$.
\end{corollary}
For simplicity, put $Cq^k=\chi(Jq^k)$. The total conjugate homomorphism is defined by $Cq=\sum_{k \ge 0} Cq^k$. The reason that $Cq$ is called a homomorphism is the following result.
\begin{proposition}
The map $Cq: \mathbb{Z}_2[\xi_1, \dots,\xi_n] \to \mathbb{Z}_2[\xi_1, \dots,\xi_n]$ is the inverse of the homomorphism $Jq: \mathbb{Z}_2[\xi_1, \dots,\xi_n] \to \mathbb{Z}_2[\xi_1, \dots,\xi_n]$.
\end{proposition}
\begin{proof}
An straightforward calculations shows that 
\[ Jq \circ Cq=\sum_{i=0}^\infty Jq(Cq^i) = \sum_{i=0}^\infty \sum_{j=0}^\infty Jq^j Cq^i=1. \]
Since $\chi*{\rm id}={\rm id}*\chi=1$, then by a similar way $Cq \circ Jq=1$. Now, $Cq$ is an algebraic homomorphism since its inverse, $Jq$, is so.
\end{proof}
The following result is immediately concluded .
\begin{corollary}[Cartan formula for the $Cq$]
For any $f,g \in \mathbb{Z}_2[\xi_1, \dots,\xi_n]$ we have $Cq(fg)=\sum_{i+j=k} Cq^i(f) Cq^j(g)$.
\end{corollary}

\section{Division ring of fractions of $\mathcal{J}_2$}\label{divrf}
We first recall the Ore condition which is a condition introduced by Oystein Ore, in connection with the question of extending the construction of a field of fractions beyond commutative rings, or more generally localization of rings. Let $R$ be a ring, not necessary commutative and let $C$ be a nonempty subset of $R$ excluding 0. Then $R$ is said to satisfy the right Ore condition with respect to $C$ if, given $a \in R$ and $c \in C$, there exist $b \in R$ and $d \in C$ such that $ad = cb$. 

It is well known that $R$ has a classical right ring of fractions  if and only if $R$ satisfies the right Ore condition with respect to $C$ when $C$ is the set of regular elements of $R$ (an element of $R$ is called a regular element if it is not a zero-divisor). It is also well known that not every ring has a classical right ring of fractions. 

In the case $C=R\setminus {0}$  we simply say that $R$ satisfies the Ore condition. In this case, $R$ has a classical right  ring of fractions if and only if $R$ has no zero divisor, that is, $R$ is a domain, called Ore domain. The left Ore condition is defined similarly.
Clearly, the ring  $\mathcal{J}_2$ is a non-commutative domain. So, the natural question is that if the ring of fractions of $\mathcal{J}_2$ exists. For an affirmative answer, $\mathcal{J}_2$ should be an Ore domain.
\begin{proposition}\label{ore}
 The ring $\mathcal{J}_2$ is an Ore domain. 
\end{proposition}
\begin{proof}
Being a domain, we show that $\mathcal{J}_2$ satisfies the Ore condition. For any $\theta, \eta \in \mathcal{J}_2$ with $\deg(\theta)=p, \deg(\eta)=q$ the equation
\[ \theta  \sum_{|I|=p+q} a_I Jq^I = \eta  \sum_{|J|=2p} a_J Jq^J \] 
has always solutions since effecting both sides to  $\xi^m$, for any $m>0$ and single variable $\xi$, gives a system of $2p+q$ linear equations in $2^{p+q-1}+2^{2p-1}$  unknowns which has always solutions.
\end{proof}
In practice, we do not need this large number of unknowns and we may solve the system of equations with less unknowns. We apply the details of the proof in the following example.
\begin{example}\label{3-2}
Let $\theta=Jq^1, \eta=Jq^2$. Then we have $p=1, q=2$. We find the unknowns $a,b,c,d,x,y$ such that the equation
\[ Jq^1(aJq^3+bJq^2Jq^1+cJq^1Jq^2+dJq^1Jq^1Jq^1)=Jq^2(xJq^2+yJq^1Jq^1) \]
has solutions.  Solving the resulted system of four linear equations in six unknowns, we get
\begin{multline*}
Jq^1 \big( (x+\frac{32}{8}y)Jq^3+(x-\frac{29}{8}y) Jq^2Jq^1+\frac{7}{4}yJq^1Jq^2 \\+(-\frac{1}{6}x+\frac{23}{24}y)Jq^1Jq^1Jq^1 \big)
                           =Jq^2(xJq^2+yJq^1Jq^1).
\end{multline*}
For example if $y=0, x \ne 0$ we have
\begin{equation}\label{eqeq}
Jq^1(Jq^3+Jq^2Jq^1-\frac{1}{6}Jq^1Jq^1Jq^1)=Jq^2Jq^2. 
\end{equation}
\end{example}
Let $a,b,c,d$ be any elements of an Ore domain $R$ with identity, such that $b \ne 0$ and $d \ne 0$. We recall the addition and multiplication of the quotients $ab^{-1}$ and $cd^{-1}$. The addition is given by
\[ ab^{-1}+cd^{-1}=(ad_1+cb_1)(bd_1)^{-1}, \]
where $bd_1=db_1$ for some $b_1,d_1 \in R$ with $d_1 \ne 0$. The multiplication is given by
\[ (ab^{-1})(cd^{-1})=(ac_1)(db_1)^{-1}, \]
where $bc_1=cb_1$ for some $b_1,c_1 \in R$ with $c_1 \ne 0$. Also we have $-(ab^{-1})=(-a)b^{-1}$. 
We give an example of addition. 
\begin{example}
Using the relation \eqref{eqeq} in Example \ref{3-2} we may write
\[
(Jq^1)^{-1}+(Jq^2)^{-1} = (Jq^3+Jq^2Jq^1-\frac{1}{6}Jq^1Jq^1Jq^1+Jq^2)(Jq^2Jq^2)^{-1}.
\]
\end{example}
Of course, $(Jq^0)^{-1}=Jq^0$. Also for all $k > 0$ we have
\[ Jq^k(Jq^k)^{-1}=(Jq^k)^{-1}Jq^k=1. \]
We write $Jq^{-k}$ for $(Jq^k)^{-1}$. 
Note that, for 
\[ Jq^i,Jq^j,Jq^{i_1},Jq^{j_1} \in \mathcal{J}_2, \text{ with } Jq^j \ne 0 \ne Jq^{j_1}, \]
from $Jq^i Jq^{j_1}=Jq^j Jq^{i_1}$ we conclude $Jq^{-j}Jq^i=Jq^{i_1}Jq^{-j_1}$.

Following \cite{sten}, we denote by $Q_{cl}^r(\mathcal{J}_2)$ and $Q_{cl}^l(\mathcal{J}_2)$ the division rings of right and left classical rings of fractions of $\mathcal{J}_2$, respectively. Note that the rings $Q_{cl}^r(\mathcal{J}_2)$ and $Q_{cl}^l(\mathcal{J}_2)$ are isomorphic. 

To proceed out the study, we extend the operations $Jq^k$ on the quotient field $\mathbb{Q}_2(\xi_1, \dots, \xi_n)$ using the homomorphism $Jq:  \mathbb{Q}_2[\xi_1, \dots, \xi_n] \to  \mathbb{Q}_2[\xi_1, \dots, \xi_n]$. We start with the total operation $Jq$, which is extended on $\mathbb{Q}_2(\xi_1, \dots, \xi_n)$ by
\[ Jq(f/g)=Jq(f)/Jq(g), \]
for any $f,g \in  \mathbb{Q}_2[\xi_1, \dots, \xi_n]$ with $g \ne 0$.
Now the operations $Jq^k$, for $k \ge 0$, are extended naturally subject to the equation $Jq=\sum_{k\ge 0}Jq^k$. For example, since for any single variable $\xi$, $Jq(\xi)=\xi+\xi^2$, we have
\[ 1=Jq(1)=Jq(\xi\frac{1}{\xi})=Jq(\xi)Jq(\frac{1}{\xi})=(\xi+\xi^2) Jq(\frac{1}{\xi}). \]
Therefor,
\[ Jq(\frac{1}{\xi})=\frac{1}{\xi+\xi^2}=\sum_{k=0}^{\infty}(-1)^k \xi^{k-1}. \]
To summarize,
\begin{proposition}\label{jqk1xi}
For $k \ge 0$ and any single variable $\xi$ we have $Jq^k(\frac{1}{\xi})=(-1)^k \xi^{k-1}$. 
\end{proposition}
The  Cartan formula can also be applied obviously. For example
\[ Jq^1(\frac{1}{\xi^2})=Jq^1(\frac{1}{\xi}\frac{1}{\xi})=\frac{1}{\xi}Jq^1(\frac{1}{\xi})+\frac{1}{\xi}Jq^1(\frac{1}{\xi})=-\frac{2}{\xi}. \]
Note that in $Q_{cl}^r(\mathcal{J}_2)$ there are elements of the form $\frac{1}{a}Jq^{-k}=(aJq^k)^{-1}$, where $a \in \mathbb{Z}_2$. Therefor $Q_{cl}^r(\mathcal{J}_2)$ is a $\mathbb{Q}_2$-algebra.

From the localization theory of rings we know that the right $\mathcal{J}_2$-module $ \mathbb{Q}_2[\xi_1, \dots, \xi_n]$ can also been localized  to a $Q_{cl}^l(\mathcal{J}_2)$-module. The resulted  $Q_{cl}^l(\mathcal{J}_2)$-module is 
\[ Q_{cl}^l(\mathcal{J}_2) \otimes_{\mathcal{J}_2} \mathbb{Q}_2[\xi_1, \dots, \xi_n] \]
containing the ring of polynomials $\mathbb{Q}_2[\xi_1, \dots, \xi_n]$. We study the structure of this $Q_{cl}^l(\mathcal{J}_2)$-module. For $k \ge 0$, $Jq^{-k}$ acts on $ \mathbb{Q}_2[\xi_1, \dots, \xi_n]$ via $Jq^k$ subject to $Jq^kJq^{-k}=Jq^{-k}Jq^k=1$. 

Note that, for $1$-variable polynomials $f$ and $g$, $Jq^k(f)=Jq^k(g)$ if and only if $f-g=a_0+a_1\xi+\cdots + a_{k-1}\xi^{k-1}$. Now assuming $Jq^{-k}(f+g)
=h, Jq^{-k}(f)+Jq^{-k}(g)=h'$ and effecting $Jq^k$ to both relations we see that $Jq^k(h)=Jq^k(h')$. Thus 
\[ Jq^{-k}(f+g)=Jq^{-k}(f)+Jq^{-k}(g)+ \overbrace{a_0+a_1\xi+\cdots + a_{k-1}\xi^{k-1}}^{\text{constant term}}. \]
Similarly, for $q \in \mathbb{Q}_2$, 
\[ Jq^{-k}(qf)=qJq^{-k}(f)+ \overbrace{b_0+b_1\xi+\cdots + b_{k-1}\xi^{k-1}}^{\text{constant term}}. \]

One sees that the linearity of the $Jq^{-k}$ is analogous to the linearity of the integration operator in the usual calculus and, like there, we also assume the $Jq^{-k}$ linear and avoid writing the constant terms. The next result evaluate the $Jq^{-k}$ in single variables.
\begin{proposition}\label{m>k}
For any $m > k$ and any single variable $\xi$ we have
\begin{equation}\label{jq-k}
Jq^{-k}(\xi^m)=\frac{1}{\binom{m-k}{k}} \xi^{m-k}. 
\end{equation}
\end{proposition}
\begin{proof}
Since $Jq^k(\xi^m)=\binom{m}{k}\xi^{m+k}$, we have
\[ \xi^m=Jq^{-k}Jq^k(\xi^m)= Jq^{-k} \Big(\binom{m}{k}\xi^{m+k} \Big)=\binom{m}{k}Jq^{-k}(\xi^{m+k}). \]
The result now comes by the linearity of $Jq^{-k}$. 
\end{proof}
\begin{remark}
As a matter of fact, a constant term should have been written in the right hand side of \eqref{jq-k}. But, as the usual calculus, we omitted the constant terms here. Nevertheless, we will write the constant term, whenever needed.
\end{remark}

The relation $Jq^kJq^{-k}=Jq^{-k}Jq^k=1$ may confuse that $Jq^{-k}$ is the inverse of the operation $Jq^k$ on $\mathbb{Q}_2[\xi_1,\dots,\xi_n]$ whereas this is not true since $Jq^k$ is not injective to have any inverse. The operation $Jq^{-k}$ has been obtained by the localization of the ring $\mathcal{J}_2$. What is true to be said is that the roles of $Jq^k$ and $Jq^{-k}$ are the same as those of the derivation and the integration operators in the usual calculus. The relation $Jq^kJq^{-k}=1$ is interpreted as if the derivation of the integration of a function is the function itself. It should be noted that, unlike the usual calculus, here the derivation-like operation $Jq^k$ increases the degree and the integration-like operation $Jq^{-k}$ decreases it. 

We turn to determine the structure of the $Q_{cl}^l(\mathcal{J}_2)$-module 
\[ Q_{cl}^l(\mathcal{J}_2) \otimes_{\mathcal{J}_2} \mathbb{Q}_2[\xi]. \] 
By Proposition \ref{jqk1xi} we know that, for any single variable $\xi$,  $Jq^1\big(\frac{1}{\xi}\big)=-1$ which implies $Jq^{-1}(1)=-\frac{1}{\xi}$. In general, for any $q \in \mathbb{Q}_2$ we have $Jq^{-1}(q)=-\frac{q}{\xi}$. Thus $\frac{1}{\xi} \in Q_{cl}^l(\mathcal{J}_2) \otimes_{\mathcal{J}_2} \mathbb{Q}_2[\xi]$. By Cartan formula it is seen that $Jq^{-1}\big(\frac{1}{\xi^m} \big)=\frac{-m}{\xi^{m-1}}$. In other words, for any $m\ge 1$,
\[ \frac{1}{\xi^m} \in Q_{cl}^l(\mathcal{J}_2) \otimes_{\mathcal{J}_2} \mathbb{Q}_2[\xi]. \]
Therefor $Q_{cl}^l(\mathcal{J}_2) \otimes_{\mathcal{J}_2} \mathbb{Q}_2[\xi]$ contains the field of fractions of $\mathbb{Q}_2[\xi]$. This is not all the story. In Section \ref{lint} we shall see that, in the linear transformation norm (Definition \ref{ltn}) we have 
\[ \|(Jq^1)^k\|_L=
\begin{cases}
\big(\frac{1}{2}\big)^\frac{k}{2}, \,	\quad \text{ if $k$ is even}\\
\big(\frac{1}{2}\big)^\frac{k-1}{2}, \ \, \text{ if $k$ is odd}\\
\end{cases}
\]
Thus the series $\sum_{k=0}^\infty (Jq^1)^k$ is convergent and its value is $(1-Jq^1)^{-1}$.
Therefor, we have
\[ (1-Jq^1)^{-1}(\xi)=\Big(\sum_{k=0}^\infty (Jq^1)^k \Big)(\xi)=\sum_{k=0}^\infty (Jq^1)^k(\xi)=\sum_{k=0}^\infty k!\xi^{k+1}, \]
which is an infinite power series. In the second equality we used the continuity of the $Jq^k$. In general, the range of $(1-Jq^1)^{-1}$ consists of the power series of the form $\sum_{n=0}^\infty a_n \xi^n$, where $a_n$ is eventually $n!$. This power series is convergent for any $\xi \in \mathbb{Q}_2$, with $|\xi|_2<1$. On the other hand, the rang of $(1-Jq^2)^{-1}$ also consists of the power series of the form 
\[ \sum_{n=2}^\infty \frac{(n-2)!}{n2^{n-2}} \xi^n \]
which is convergent only for $\xi=0$. 

The following result follows easily by an induction.
\begin{proposition}\label{x-x0}
For any $k \ge 1$ and any $n \ge k$,
\[ Jq^k((\xi-\xi_0)^n)=\binom{n}{k} \xi^{2k} (\xi-\xi_0)^{n-k}. \]
\end{proposition}
Proposition \ref{m>k} determines $Jq^{-k}(\xi^m)$ in the case $m>k$. The next result tackles the case $m=k$. The proof is by effecting $Jq^1$ to both sides.
\begin{proposition} 
For any single variable $\xi$ we have
\[ Jq^{-1}(\xi)=\sum_{n=1}^\infty \frac{(-1)^n}{n}(\xi-1)^n \]
\end{proposition}
However, $\sum_{n=1}^\infty \frac{(-1)^n}{n} (\xi-1)^n$ is not a power series at zero. An interesting fact is that this series is the solution of a kind of differential equation $Jq^k(\zeta)=\xi^k$.
\begin{definition}
A dyadic Steenrod (ordinary) differential equation, or simply a {$\Sigma$}ODE, is an equation of the form
\[ a_k\theta_k(\zeta)+\cdots+a_1\theta_1(\zeta)+a_0 \zeta=b(\xi), \]
where for $0 \le i \le k$, $a_i=a_i(\xi)$ and $b(\xi)$ are polynomials in $\mathbb{Q}_2[\xi]$ and $\theta_1, \dots, \theta_k \in \mathcal{J}_2$. Following the methods of differential algebraic geometry, we call the variable $\zeta$, a differential indeterminate.
\end{definition}
The next straightforward result gives a sample dyadic Steenrod differential equation. 
\begin{theorem}
For any $k \ge 1$, the dyadic Steenrod differential equation $Jq^k(\zeta)=\xi^k$ has the following solution.
\[ \zeta(\xi)=\sum_{n=0}^\infty \frac{(-1)^n}{n}(\xi-1)^n. \]
\end{theorem}
The next result establishes a close relationship between the elements of $Q_{cl}^l(\mathcal{J}_2) \otimes_{\mathcal{J}_2} \mathbb{Q}_2[\xi]$ and the solutions of the {$\Sigma$}ODE with constant coefficients.
\begin{theorem}\label{sdecc}
Each {\rm {$\Sigma$}ODE} with constant coefficients has a solution in $Q_{cl}^l(\mathcal{J}_2) \otimes_{\mathcal{J}_2} \mathbb{Q}_2[\xi]$ and, conversely, each element in $Q_{cl}^l(\mathcal{J}_2) \otimes_{\mathcal{J}_2} \mathbb{Q}_2[\xi]$ satisfies in a {\rm {$\Sigma$}ODE} with constant coefficients.
\end{theorem}
\begin{proof}
Let 
\begin{equation}\label{aktta}
a_k\theta_k(\zeta)+\cdots+a_1\theta_1(\zeta)+a_0 \zeta=b(\xi) 
\end{equation}
 be a {$\Sigma$}ODE, where each $a_i$ is an element of $\mathbb{Q}_2$ and $b(\xi)$ is a polynomial with coefficients in $\mathbb{Q}_2$. We may write \eqref{aktta} as 
\[ (a_k\theta_k+\cdots+a_1\theta_k+a_0 Jq^0)(\zeta)=b(\xi). \]
Assuming $\theta=a_k\theta_k+ \cdots+a_1 \theta_1+a_0$ we have $\theta(\zeta)=b(\xi)$ or, $\zeta=\theta^{-1}(b(\xi))$. Since $b(\xi) \in \mathbb{Q}_2[\xi]$, so $\zeta \in Q_{cl}^l(\mathcal{J}_2) \otimes_{\mathcal{J}_2} \mathbb{Q}_2[\xi]$.

Conversely, any element $\zeta \in  Q_{cl}^l(\mathcal{J}_2) \otimes_{\mathcal{J}_2} \mathbb{Q}_2[\xi]$ is of the form 
\[ \delta_1(f_1)+\delta_2(f_2)+\cdots+\delta_n(f_n)=\zeta, \]
where for $1 \le i \le n$, $f_i \in \mathbb{Q}_2[\xi]$ and $\delta_i \in Q_{cl}^l(\mathcal{J}_2)$. Multiplying both sides by common denominators, we obtain a {$\Sigma$}ODE with constant coefficients with $\zeta$ as a solution.
\end{proof}
\begin{example}
Each of the series 
\[ \sum_{m=0}^\infty m!\xi^{m+1} \text{ and } \sum_{n=2}^\infty \frac{(n-2)!}{n2^{n-2}} \]
are, respectively, the solutions of the {$\Sigma$}ODE's 
\[ -Jq^1(\zeta)+\zeta=\xi \text{ and } -Jq^2(\zeta)+\zeta=\xi. \]
\end{example}

In solving the {$\Sigma$}ODE's, as in differential algebraic geometry, the convergence of the series is not of worthwhile in itself. What is worthwhile is the infinite formal series themselves. For example, attending the convergence, one of the solutions of the {$\Sigma$}OED $Jq^1(\zeta)-\zeta=0$ is the series $\sum_{n=0}^\infty\frac{1}{n!\xi^n}$ which is convergent for $\xi \in \mathbb{C}_2$, the dyadic complex numbers, provided that $|\xi|_2>1$ but, this solution is not in $Q_{cl}^l(\mathcal{J}_2) \otimes_{\mathcal{J}_2} \mathbb{Q}_2[\xi]$ since the operation $Jq^{-k}$ increases the degree. On the other hand, solving the $\zeta(\xi)=\sum_{n=0}^\infty a_n(\xi-\xi_0)^n$, by techniques of the usual differential equations, it is seen that the solutions of this {$\Sigma$}ODE are elements of $Q_{cl}^l(\mathcal{J}_2) \otimes_{\mathcal{J}_2} \mathbb{Q}_2[\xi]$ of the form $\sum_{m=0}^\infty a_n(\xi-\xi_0)^n$ in which the common term $a_n$ is given by the recursive relations
\begin{multline*} a_1=\frac{a_0}{\xi_0^2}, \ a_2=\frac{-(2\xi_0-1)}{2\xi_0^4},\\ 
\text{ and for $n \ge 2$, } (n-1)a_{n-1}+(2n\xi_0-1)a_n+\xi_0^2(n+1)a_{n+1}=0. 
\end{multline*}
An immediate consequence is that $\xi_0 \ne 0$, showing that this {$\Sigma$}ODE has no solution at $\xi_0=0$. As seen, for any $0 \ne \xi_0 \in \mathbb{Q}_2$, there is a solution for this {$\Sigma$}ODE, which shows that the {$\Sigma$}ODE $Jq^1(\zeta)=\zeta$ has infinitely many solutions. We do not know more about the independence of these solutions.

\section{Linear transformation norm}\label{lint}
Clearly, each $\theta \in \mathcal{J}_2$ is a linear operator on $\mathbb{Q}_2[\xi_1, \dots, \xi_n]$ (or $\mathbb{Z}_2[\xi_1, \dots, \xi_n]$). Recall the definition \ref{ltn} of linear transformation norm for $\theta \in \mathcal{J}_2$:
\[ \| \theta \|_L = \inf \{ c \ge 0 \mid \forall f \in \mathbb{Q}_2[\xi_1, \dots, \xi_n], \|\theta(f)\|_2 \le c\|f\|_2 \}. \]
\begin{example}\label{jq12}
As seen in Corollary \ref{boundcor}, for $k \ge 0$, we have $\| Jq^k \|_L=1$. Moreover, since
\[ Jq^1(\xi_1^{j_1}\xi_2^{j_2} \cdots \xi_n^{j_n})=\sum_{k=1}^n j_k \xi_1^{j_1} \cdots \xi_k^{j_k+1} \cdots \xi_n^{j_n}, \]
then we have
\[ Jq^1Jq^1(\xi_1^{j_1}\xi_2^{j_2} \cdots \xi_n^{j_n})=\sum_{k=1}^n j_k(j_k+1)\xi_1^{j_1} \cdots \xi_k^{j_k+2} \cdots \xi_n^{j_n}, \]
from which we conclude 
\[ \|Jq^1Jq^1(\xi_1^{j_1}\xi_2^{j_2} \cdots \xi_n^{j_n})\|_2 = \max_{1 \le k \le n} \{ |j_k(j_k+1)|_2 \ \|\xi_1^{j_1} \cdots \xi_k^{j_k+1} \cdots \xi_n^{j_n}\|_2 \}. \]
Since $j_k(j_k+1)$ is always even, we have $\|Jq^1Jq^1\|_L=\frac{1}{2}$. 
\end{example}

Some of the properties of the norm $\| \ \|_L$ are as follows.
\begin{proposition}
\begin{itemize}
\item[\rm{i)}] For any $a \in \mathbb{Q}_2$ we have $\|a\|_L=|a|_2$;
\item[\rm{ii)}] For any $\theta_1, \theta_2 \in \mathcal{J}_2$ we have $\|\theta_1 \theta_2\|_L \le \|\theta_1 \|_L \| \theta_2\|_L$. In general, the equality is not hold as $\|Jq^1Jq^1\|_L=\frac{1}{2} \ne 1 = \|Jq^1\|_L \|Jq^1\|_L$;
\item[\rm{iii)}] The norm $\| \ \|_L$ is a non-Archimedean norm. In other words, 
\[ \| \theta_1 + \theta_2 \|_L \le \max \{ \|\theta_1\|_L \, , \|\theta_2 \|_L\}. \]
\end{itemize}
\end{proposition}
\begin{proof}
The first two parts are clear from the definition of $\| \ \|_L$. For the last part, since $\| \ \|_2$ is non-Archimedean on $\mathbb{Q}_2[\xi_1, \dots, \xi_n]$ then, for any $f \in \mathbb{Q}_2[\xi_1, \dots, \xi_n]$ we have
\[ \| (\theta_1 + \theta_2)(f) \|_2 = \| \theta_1(f) + \theta_2(f) \|_2 \le \max \{ \|\theta_1(f)\|_2 \ , \|\theta_2(f) \|_2\}. \]
On the other hand, for $i=1,2$,  $\| \theta_i(f) \|_2 \le \| \theta_i \|_L \ \| f \|_2$ which implies 
\[ \| (\theta_1 + \theta_2)(f) \|_2  \le \max \{ \|\theta_1\|_L \ , \|\theta_2 \|_L\}  \|f\|_2. \]
The result is now hold since $f$ is arbitrary.
\end{proof}
The next interesting result assimilate the $\ker(\phi)$ to the unit ball in the linear transformation norm.

\begin{theorem}\label{bk}
Let $B_L(0,1)=\{ \theta \in \mathcal{J}_2 \mid \| \theta \|_L<1\}$ be the unit ball in the norm $\| \ \|_L$. Then $B_L(0,1)=\ker(\phi : \mathcal{J}_2 \to \mathcal{A}_2)$.
\end{theorem}
\begin{proof}
If $\theta \in B_L(0,1)$, then $\| \theta \|_L<1$ or, precisely, $\| \theta \|_L \le \frac{1}{2}$. For any $f \in \mathbb{Z}_2[\xi_1, \dots, \xi_n]$ we have
\[ \| \theta(f) \|_2 \le \| \theta \|_L \| f \|_2 \le \frac{1}{2}\| f \|_2 = \|2f \|_2 \]
This means that all coefficients of the polynomial $\theta(f)$ are even. It follows that $\theta \in \ker(\phi)$. We proved $B_L(0,1) \subset \ker(\phi)$. The direction of the demonstration is recursive.
\end{proof}
Theorem \ref{bk} says that $\| \theta \|_L<1$ (precisely, $\| \theta \|_L \le \frac{1}{2}$) if and only if $\theta \in \ker(\phi)$ or, equivalently, $\| \theta \|_L=1$ if and only if $\theta \notin \ker(\phi)$. Therefor we have the following corollary. By an admissible mononial we mean a monomial $\theta=Jq^A$ with the exponent vector $A=(a_1,a_2,\dots)$, where $a_i \ge 2a_{i+1}$, for $i \ge 1$. We call $\theta$ reversed admissible  if $a_i \le 2a_{i+1}$, for $i \ge 1$.

\begin{corollary}\label{admiss}
\begin{itemize}
\item[\rm{i)}] For any $k \ge 0$ we have $\|Jq^k\|_L=1$; (compare with Corollary $\ref{boundcor}$)
\item[\rm{ii)}] If $\theta \in \mathcal{J}_2$ is admissible or reversed admissible, then $\|\theta\|_L=1$;
\item[\rm{iii)}] If for $a \le 2b$ we put
\[ R(a,b)=Jq^aJq^b-\sum_{j=0}^{[a/2]} \binom{b-j-1}{a-2j} Jq^{a+b-j}Jq^j, \]
then $\| R(a,b)\|_L\le \frac{1}{2}$;
\item[\rm{iv)}] For any $k \ge 1$ there is a positive integer $m(k)$ such that 
\[ \|(Jq^k)^{m(k)}\|_L \le \frac{1}{2}. \]
In fact, $m(k)$ is the nilpotency degree of $Sq^k$.
\end{itemize}
\end{corollary}
The next corollary analyzes the topology induced by the norm $\| \ \|_L$.
\begin{corollary}\label{rm}
The topology induced by the norm $\| \ \|_L$ on $\mathcal{J}_2$ is the $\ker(\phi)$-adic topology.
\end{corollary}
Using Corollary \ref{rm} we can do some calculations. In Example \ref{jq12} we saw that $\| (Jq^1)^2 \|_L=\frac{1}{2}$. That is, $(Jq^1)^2 \in \ker(\phi)$. Also, $\|(Jq^1)^3\|_L=\frac{1}{2}$ since
\[ (Jq^1)^3 \in \ker(\phi) \setminus (\ker(\phi))^2. \]
But, 
\[(Jq^1)^4 \in (\ker(\phi))^2 \setminus (\ker(\phi))^3. \]
Thus $\|(Jq^1)^4\|_L=\frac{1}{2^2}=\frac{1}{4}$. Continuing this process, we see by induction that
\[ \|(Jq^1)^k\|_L=
\begin{cases}
\big(\frac{1}{2}\big)^\frac{k}{2}, \,	\quad \text{ if $k$ is even}\\
\big(\frac{1}{2}\big)^\frac{k-1}{2}, \ \, \text{ if $k$ is odd}\\
\end{cases}
\]
As a matter of fact, for other operations $Jq^k$, $k \ge 2$, decreasing the norm depends on the nilpotency degree of $Sq^k$. For example 
\[ \|Jq^2\|_L=\|(Jq^2)^2\|_L=\|(Jq^2)^3\|_L=1, \]
while 
\[\|(Jq^2)^4\|_L=\|(Jq^2)^5\|_L=\|(Jq^2)^6\|_L=\|(Jq^2)^7\|_L=\frac{1}{2}, \]
as $(Sq^2)^4=0$.

In non-Archimedean norms, the so-called strong triangle inequality $\| \theta_1 + \theta_2 \|_L \le \max \{ \|\theta_1\|_L \ , \|\theta_2 \|_L\}$ turns to equality if $\| \theta_1\|_L \ne \| \theta_2 \|_L$. For example, for any $k \ge 0$ we have $\| Jq^1-Jq^1Jq^{2k+1} \|_L=1$ since $\| Jq^1 \|_L =1 > \| Jq^1Jq^{2k+1} \|_L$. However, the case $\| \theta_1\|_L = \| \theta_2 \|_L$ may give the equality in the strong triangle inequality or may not.  For example,  for any $k > 0$, $\|1-Jq^k\|_L=1$ (since $1-Jq^k \notin \ker(\phi)$) meanwhile $\|1\|_L=\|Jq^k\|_L=1$. On the other hand $\|Jq^3\|_L=\|Jq^1Jq^2\|_L=1$ while, 
\[ \|Jq^3-Jq^1Jq^2\|_L=\|R(1,2)\|_L \le \frac{1}{2} < \max\{ \|Jq^3\|_L, \|Jq^1Jq^2\|_L  \}. \]
\begin{remark}
The linear transformation norm on $\mathcal{A}_2$ is trivial: $\|0\|_L=0$ and $\|\theta\|_L=1$, for any $\theta \ne 0$ in $\mathcal{A}_2$. Thus, the topology induced by this norm is the discrete topology, the finest topology under which the homomorphism $\phi: (\mathcal{J}_2, \| \ \|_L) \to \mathcal{A}_2$ is continuous.
\end{remark}

Now that the topology induced by the norm $\| \ \|_L$ on $\mathcal{J}_2$ is the $\ker(\phi)$-adic topology, we can easily exhibit a completion $(\widehat{\mathcal{J}_2})_L$  of $\mathcal{J}_2$ as the inverse limit of the system
\[ \cdots \to \mathcal{J}_2/B_L(0,\frac{1}{4}) \to \mathcal{J}_2/B_L(0,\frac{1}{2}) \to \mathcal{J}_2/B_L(0,1)=\mathcal{A}_2, \]
where for $i \ge 0$, the homomorphism $\mathcal{J}_2/B_L(0,\frac{1}{2^{i+1}}) \to \mathcal{J}_2/B_L(0,\frac{1}{2^i})$ is natural. 
In other words, 
\[ (\widehat{\mathcal{J}_2})_L=\varprojlim_k \mathcal{J}_2 / B_L(0,\frac{1}{2^k}) \]
is the set of infinite sequences $(\theta_0,\theta_1,\theta_2,\dots)$, where for any $k \ge 0, \theta_k \in B_L(0,\frac{1}{2^k})$. We shall use the standard  notation $\sum_{k=0}^\infty \theta_k$ for the elements of $(\widehat{\mathcal{J}_2})_L$. We study the nature of the elements of $(\widehat{\mathcal{J}_2})_L$. Of course $\mathcal{J}_2 \subset (\widehat{\mathcal{J}_2})_L$. There are other elements in this algebra. To recognize them, we need a result from non-Archimedean analysis \cite{bosch}.
\begin{proposition}\label{conv}
In a complete non-Archimedean space, the series $\sum_{n=0}^\infty a_n$ is convergent if and only if $\lim_{n \to \infty} a_n=0$.
\end{proposition} 
As an example, for any $k \ge 1$, since $\lim_{n \to \infty} (Jq^k)^n=0$ then, the series $\sum_{n=0}^\infty (Jq^k)^n$  is convergent and hence is an element of $(\widehat{\mathcal{J}_2})_L$. With an straightforward calculations it is seen that 
\[ \sum_{n=0}^\infty (Jq^k)^n=(1-Jq^k)^{-1}. \]
Therefore, the operation $1-Jq^k$ is invertible in $(\widehat{\mathcal{J}_2})_L$. On the invertibility of the elements of $(\widehat{\mathcal{J}_2})_L$ we have the next result.
\begin{theorem}
Let $K$ be the kernel of the homomorphism 
\[ (\widehat{\mathcal{J}_2})_L \to \mathcal{J}_2/B_L(0,1)=\mathcal{A}_2. \]
Then any element in $(\widehat{\mathcal{J}_2})_L$ of the form $1+\theta$ for some $\theta \in K$ is invertible in $(\widehat{\mathcal{J}_2})_L$. 
Moreover, there is a one-to-one correspondence between the set of maximal ideals of $(\widehat{\mathcal{J}_2})_L$ and the set of ideals of $\mathcal{A}_2$.
\end{theorem}
Now we give a completion of the $\mathcal{J}_2$-module $\mathbb{Z}_2[\xi_1, \dots, \xi_n]$ as the inverse limit of the system
\begin{multline*} 
\cdots \to \mathbb{Z}_2[\xi_1, \dots, \xi_n]/B_L(0,\frac{1}{4})\mathbb{Z}_2[\xi_1, \dots, \xi_n] \\
\to \mathbb{Z}_2[\xi_1, \dots, \xi_n]/B_L(0,\frac{1}{2})\mathbb{Z}_2[\xi_1, \dots, \xi_n] \\
\to \mathbb{Z}_2[\xi_1, \dots, \xi_n]/B_L(0,1)\mathbb{Z}_2[\xi_1, \dots, \xi_n] 
\end{multline*}
of abelian groups. We consider the elements of this inverse limit which are power series of the form 
\[ \sum_{n=0}^\infty a_n, \quad a_n \in B_L(0,\frac{1}{2^n})\mathbb{Z}_2[\xi_1, \dots, \xi_n]. \]
First of all note that, for any $1 \le i \le n$ and any $k_i \ge 0$, $\|\xi_i^{k_i}\|_L=1$ since $\xi_i^{k_i}$ is in the range of no element of $\ker(\phi)$. Of course whenever $k_i$ is even we have $\xi_i^{k_i}=Jq^{k_i/2}(\xi_i^{k_i/2})$. With these all, still $Jq^{k_i/2} \notin \ker(\phi)$. Now we find the value of the norm of $a\xi^i$, where $a \in \mathbb{Z}_2$. Let first see an example.
\begin{example}\label{4!}
Suppose that we want to calculate the value of $\|4!\xi^5\|_L$. We have
\begin{align*}
4!\xi^5 &= Jq^1(3!\xi^4)\\
        &= Jq^1Jq^1(2\xi^3) \text{ (then $4!\xi^5 \in (\ker(\phi))\mathbb{Z}_2[\xi])$}\\
				&= Jq^1Jq^1Jq^1Jq^1(\xi) \text{ (then $4!\xi^5 \in (\ker(\phi))^2\mathbb{Z}_2[\xi])$)}.
\end{align*}
From the above calculations we conclude that $\|4!\xi^5\|_L \le \big(\frac{1}{2}\big)^2$. This value may be reduced more. Note that $a=aJq^0 \in \ker(\phi)$ if and only if $|a|_2 <1$. Hence we have also $4!\xi^5 = 4!Jq^0(\xi^5)$. Then $4!\xi^5 \in (\ker(\phi))^3\mathbb{Z}_2[\xi]$. 
\end{example}
Therefor, as explained in Example \ref{4!}, we have $\|a\xi^k\|_L=|a|_2$ and it is seen that in the series $\sum_{n=0}^\infty a_n\xi^n \in \widehat{\mathbb{Z}_2[\xi]}_L$ we have $\lim_{n \to \infty} a_n=0$ since the elements of $\widehat{\mathbb{Z}_2[\xi]}_L$ are the sequences of the form $(f_0, f_1, f_2, \dots)$, where $f_i \in (\ker(\phi))^i\mathbb{Z}_2[\xi]$. Thus we have the following result.
\begin{theorem}
$\widehat{\mathbb{Z}_2[\xi]}_L=T_1(\mathbb{Z}_2)$, where $T_1(\mathbb{Z}_2)$ is the Tate algebra over $\mathbb{Z}_2$.
\end{theorem}
The above discussion is, of course, hold for higher variables. That is, for $n \ge 1$, 
\[ \widehat{\mathbb{Z}_2[\xi_1,\dots,\xi_n]}_L=T_n(\mathbb{Z}_2). \]

Recall that in rigid analytic geometry (see \cite{bosch}), the studies starts with the algebras
\[ T_n=T_n(R)= \Big\{ \sum a_J\xi^J \Bigm| {|a_J| \to 0} \text{ as } |J| \to \infty \Big\} \]
called the Tate algebras, where $R$ is a nontrivial complete non-Archimedean ring, $|a_J|$ means the non-Archimedean norm of $a_J$ in $R$ (here, $| a_J |_2$ in $\mathbb{Z}_2$) and $|J|=j_1+\cdots+j_n$. Here we used the notation \eqref{mi} before Definition \ref{nan}.

An element $f=\sum a_J\xi^J\in T_n(K)$, where $K$ is a complete non-Archimedean field, is a function on 
\[ B^n(K^{\rm alg})=\left\{(x_1,\ldots,x_n)\in (K^{\rm alg})^n \Big| \ |x_i|\leq1 \text{ for any }i=1,\ldots,n \right\}, \]
where $K^{\rm alg}$ stands for the algebraic closure of $K$. Note from Theorem 5.7 that $f$ converges not only for any $\xi\in B^n(K)$, but also for $\xi\in B^n(L)$, where $L$ is any extension of $K$. This is why the elements of $T_n(R)$ are called strictly convergent power series. Any maximal ideal $\mathfrak{m}\subset T_n(R)$ is of the form
\[\mathfrak{m}=\mathfrak{m}_{(x_1, \ldots, x_n)}=\left\{f\in T_n(K)\middle| f(x_1,\ldots,x_n)=0 \right\} \]
for an appropriate $x=(x_1,\ldots,x_n)\in B^n(K^{\rm alg})$ \cite[Chapter 5]{bosch}. The map $\mathfrak{m}_x\mapsto x$
is in general a surjection, however, it turns to an injection only if $K$ is algebraically closed.

Affinoid algebras and their collections of maximal ideals, known as affinoid spaces, are essential tools in the studying the rigid analytic geometry. Recall from \cite{bosch} that any quotient ring $A=T_n/I$ of a Tate algebra $T_n$ is said to be an affinoid algebra and the collection $\operatorname{Max} A$ of its maximal ideals, denoted by $\operatorname{Sp}(A)$, is the corresponding affinoid space.

By effecting the elements of $\widehat{(\mathcal{J}_2)}_L $ to an ideal of $T_n$, the action of $Jq^k$, $k \ge 0$, may be extended to affinoid spaces. For example, $\operatorname{Sp}\big(T_2(\mathbb{Q}_2)/(2\xi_2-\xi_1^2)\big) $ is an affinoid space, called closed ball of radius $2^{-1/2}$. We know $Jq^1(2\xi_2-\xi_1^2)=2\xi_2^2-2\xi_1^3$. Thus $Jq^1$ maps this closed ball to an affinoid space corresponding to an elliptic curve. 

\section{Adem norm and the hit problem}\label{ademn}
The two-sided ideal $\mathcal{J}_2^+=\langle Jq^1,Jq^2 \rangle$ of the $\mathbb{Q}_2$-algebra $\mathcal{J}_2$ is a maximal ideal and the $\mathcal{J}_2^+$-adic norm can be defined on $\mathcal{J}_2$ which we rename it the Adem norm. 
\begin{definition}
The Adem norm $\| \ \|_A$ is defined on the $\mathbb{Q}_2$-algebra $\mathcal{J}_2$, as follows.  
\[ \|Jq^1\|_A=\|Jq^2\|_A=\frac{1}{2} \text{ and for any $q \in \mathbb{Q}_2$, } \|q\|_A=1. \]  
\end{definition}
The proof of the next result shows the reason of this naming. 
\begin{theorem}\label{anorm}
For any $k \ge 2$, $\|Jq^k\|_A=\big( \frac{1}{2} \big) ^{k-1}$
\end{theorem}
\begin{proof}
We should find a maximal power of $\mathcal{J}_2^+$ containing $Jq^k$. Of course $Jq^k \in \mathcal{J}_2^+$. For $k=2$, by definition we have $\|Jq^2\|_A=\frac{1}{2}$. Also for $k=3$ by the Adem expansion of $A_3$ we see that $Jq^3 \in (\mathcal{J}_2^+)^2 \setminus (\mathcal{J}_2^+)^3$ which shows $\|Jq^3\|_A=\big( \frac{1}{2} \big) ^2$.

On the other hand, using the dyadic Adem expansion, Theorem \ref{ademk}, for $k \ge 4$, there are coefficients $a_1, a_2, \dots, a_k$, not all zero, such that 
\[ a_k Jq^k+\sum_{i=1}^{k-1} a_iJq^iJq^{k-i}=0. \]
This shows that  $Jq^k \in (\mathcal{J}_2^+)^2$. We are looking for the greatest $j$ for which  $Jq^k \in (\mathcal{J}_2^+)^j$. Consider the equation
\[ aJq^k+b(Jq^1)^k+\sum a_IJq^I=0, \]
where $I$ runs over the $(k-1)$-partitions of $k$. Effecting both sides to $\xi^m$, for any $m>0$ and any single variable $\xi$, gives a polynomial of degree $k-1$. This leads to a system of $k$ linear equations in $k+1$ unknowns which has always nonzero solutions. Thus  $Jq^k \in (\mathcal{J}_2^+)^{k-1}$. On the other hand, the strict $k$-partition $b(Jq^1)^k$ can not expand $Jq^k$, showing that $Jq^k \notin (\mathcal{J}_2^+)^k$. Hence $\|Jq^k\|_A=\big( \frac{1}{2} \big) ^{k-1}$, as required.
\end{proof}
Since $\mathcal{J}_2$ has no zero divisors, we conclude,
\begin{corollary}
The Adem norm $\| \ \|_A$ is a multiplicative norm. More precisely, for any $\theta_1, \theta_2 \in \mathcal{J}_2$, we have $\| \theta_1 \theta_2 \|_A=\| \theta_1 \|_A\| \theta_2 \|_A$.
\end{corollary}
With the Adem norm $\| \ \|_A$, $\mathcal{J}_2$ is not complete since the series $\sum_{k=0}^\infty Jq^k$ is a non-convergent Cauchy series. The completion of this ring is the inverse limit of the system 
\[ \cdots \to \mathcal{J}_2/(\mathcal{J}_2^+)^3 \to \mathcal{J}_2/(\mathcal{J}_2^+)^2 \to \mathcal{J}_2/\mathcal{J}_2^+ \cong \mathbb{Q}_2 \] 
which is denoted by $(\widehat{\mathcal{J}_2})_A$. The elements of $(\widehat{\mathcal{J}_2})_A$ are the infinite series 
\[ \sum_{k=0}^\infty a_k \theta_k, \ a_k \in \mathbb{Q}_2,  \ \theta_k \in (\mathcal{J}_2^+)^k. \]
In $(\widehat{\mathcal{J}_2})_A$ there are elements of the form $\sum_{k=0}^\infty a_kJq^k$ which coincide with the total operation $Jq$ whenever $a_k=1$ for all $k \ge 0$. In general, 
\[ \sum_{k=0}^\infty a_kJq^k: \mathbb{Q}_2[\xi_1,\dots,\xi_n] \to \mathbb{Q}_2[\xi_1,\dots,\xi_n] \]
is a linear operator but is not always a ring homomorphism. In fact, the necessary and sufficient condition for the equation
\[ \sum_{k=0}^\infty a_kJq^k(fg)=\big( \sum_{k=0}^\infty a_kJq^k(f) \big) \big( \sum_{k=0}^\infty a_kJq^k(g) \big) \]
to be true for all $f,g \in \mathbb{Q}_2[\xi_1,\dots,\xi_n]$ is that, for all $k \ge 0$, we have
\begin{equation}\label{gcf}
a_kJq^k(fg)=\sum_{i+j=k}a_ia_jJq^i(f)Jq^j(g),
\end{equation}
which is a generalization of Cartan formula. One of the solutions of \eqref{gcf} is $a_k=a_1^k$ for $k \ge 0$ (in particular, $a_0=a_1^0=1$). Thus, the elements 
\[ \psi_q=\sum_{k=0}^\infty q^kJq^k, \ q \in \mathbb{Q}_2 \]
of $(\widehat{\mathcal{J}_2})_A$ are homomorphisms on $\mathbb{Q}_2[\xi]$ with inverses 
\[ \psi_q^{-1}=\sum_{k=0}^\infty q^k \chi(Jq^k). \]
In other words, the elements $\psi_q$ are automorphisms of $\mathbb{Q}_2[\xi_1,\dots,\xi_n]$.

Given a graded module $M$, there is a degree norm $\|x\|_\rho = \rho^{\deg x}$, for any $x \in M$ and $0<\rho<1$. Applying this norm on $\mathcal{J}_2$, we get the norm $\|\theta\|_\rho=\rho^{\deg \theta}$ for $\theta \in \mathcal{J}_2$. Putting $\rho=\frac{1}{2}$ 
, it is easily seen that $\frac{1}{2}\| \ \|_\rho \le \| \ \|_A \le \| \ \|_\rho$. Thus we have
\begin{proposition}
The Adem norm and the degree norm induce the same topology on $\mathcal{J}_2$.
\end{proposition}
Therefor, we can say that the expansions by $(k-1)$-partitions are, somehow, the maximal relations over $\mathcal{J}_2$ and the Adem norm is the most natural norm defining on $\mathcal{J}_2$.
\begin{remark}
Regarding the homomorphism $\phi:\mathcal{J}_2 \to \mathcal{A}_2$ for which $\ker(\phi)$ annihilated in $\mathcal{J}_2/\ker(\phi)$, it is seen that the Adem norm $\| \ \|_A$ may also be defined on $\mathcal{A}_2$ and the induced topology is also the $\mathcal{A}_2^+$-adic topology. However, unlike $\mathcal{J}_2$, the Adem norm in $\mathcal{A}_2$ is not multiplicative. For example $\|Sq^1Sq^1\|_A=0$, while $\|Sq^1\|_A=\frac{1}{2}$.
\end{remark}
We now restrict our attention to the $\mathbb{Z}_2$-module $\mathcal{J}_2$ and, by a similar way as in the linear transformation norm, complete the $\mathcal{J}_2$-module $\mathbb{Z}_2[\xi_1,\dots,\xi_n]$. Via
\begin{multline*} 
\cdots \to \mathbb{Z}_2[\xi_1, \dots, \xi_n]/B_A(0,\frac{1}{4})\mathbb{Z}_2[\xi_1, \dots, \xi_n] \\
\to \mathbb{Z}_2[\xi_1, \dots, \xi_n]/B_A(0,\frac{1}{2})\mathbb{Z}_2[\xi_1, \dots, \xi_n] \\
\to \mathbb{Z}_2[\xi_1, \dots, \xi_n]/B_A(0,1)\mathbb{Z}_2[\xi_1, \dots, \xi_n] 
\end{multline*}
it is easily seen that completion of $\mathbb{Z}_2[\xi_1,\dots,\xi_n]$, in the Adem norm, is $\mathbb{Z}_2\llbracket\xi_1,\dots,\xi_n\rrbracket$, the $n$-variable formal power series. Surprisingly, given $f \in \mathbb{Z}_2\llbracket\xi_1,\dots,\xi_n\rrbracket$ the value of $\|f\|_A$ is concerned with the hit problem. For more details on the hit problem see \cite{criterion,kameko,mothebe,peterson,sum}.
\begin{definition}
An element $f \in \mathbb{Z}_2\llbracket\xi_1,\dots,\xi_n\rrbracket$ is said to be hit if there is a finite sum
\[ f=\sum_{i>0} Jq^i(f_i), \]
with $f_i \in \mathbb{Z}_2\llbracket\xi_1,\dots,\xi_n\rrbracket$ in positive degree.
\end{definition}
The next result, which is clear from the above discussion, gives a criterion for an elements of $\mathbb{Z}_2\llbracket\xi_1,\dots,\xi_n\rrbracket$ to be hit.
\begin{theorem}\label{hit}
The element $f \in \mathbb{Z}_2\llbracket\xi_1,\dots,\xi_n\rrbracket$ is hit if and only if $\|f\|_A<1$. In other words the set of hit elements is 
\[ H(n)=B_A(0,1)\mathbb{Z}_2\llbracket\xi_1,\dots,\xi_n\rrbracket. \]
\end{theorem} 
Theorem \ref{hit} says that to find the generators of the cohit $\mathbb{Z}_2$-module $Q(n)$ we are not compelled to examine the elements individually. 

We determine the elements of $Q(1)$ over $\mathbb{Z}_2[\xi]$. Degree one is straightforward. For any $a \in \mathbb{Z}_2$, the polynomial $a\xi$ can not be hit. On the other hand, in degree two, $\xi^2$ and in general $a\xi^2$ for any  $a \in \mathbb{Z}_2$ is hit. In degree three we have $2\xi^3=Jq^1(\xi^2)$ is hit while $\xi^3$ is not.

The non-hit elements in degree 7 are $\xi^7$, $2\xi^7$, $3\xi^7$, while $4\xi^7=Jq^3(\xi^4)$ is hit. In general, 
\begin{proposition}
For any $n \ge 0$, $2^n\xi^{2^{n+1}-1}=Jq^{2^n-1}(\xi^{2^n})$. Hence, we have
\[ Q(1)=\bigoplus_{d=2^n-1} Q^d(1), \]
where $Q^d(1) \cong C_n$, the cyclic group of order $n$, and  $Q^1(1)=\mathbb{Z}_2$. 
\end{proposition}
Note that on $Q(n)$ we can define the quotient norm 
\[ \|f+H(n)\|=\inf_{h \in H(n)} \|f-h\|, \]
which makes $Q(n)$ to be normed but, unfortunately, with trivial norm. We cite \cite{wa:woo} for the applications of $Q^d(n)$ in the representation theory.

\section{Closing comments}

In this section we take the reader to the future of the subject by stating some comments, hints, and open problems. Because of the novelty of the subject, there are more problems than ones stated here. The hints are just our experience as well as understanding of the concepts.
\begin{problem}
Generalize the mod $p$ Steenrod algebra for $p$ odd by a similar methods.
\end{problem}
The main problem is to generalize Bockstein homomorphism in an appropriate manner. One way to do this, is using the diagram
\[
\begin{CD}
0  @>>> \mathbb{Z}_p @>>> \mathbb{Z}_{p^2} @>>> \mathbb{Z}_p @>>> 0\\
@.  @AAA @AAA @AAA  \\
0 @>>> \mathbb{F}_p @>>> \mathbb{F}_{p^2} @>>> \mathbb{F}_p @>>> 0
\end{CD}
\]
where $\mathbb{Z}_{p^2}$ is the set of $p^2$-adic integers. For an introduction to $g$-adic numbers, where $g$ is not necessarily a 
prime, see \cite{mahler}.
\vspace{0.2cm}

In Section \ref{diadsa} we saw that, by chance, ${\rank}(\mathcal{J}_2^i)=i$ for $i=1,2,3$. A natural question posed is that what the ${\rank}(\mathcal{J}_2^i)$ is for $i>3$. As a matter of fact, ${\rank}(\mathcal{J}_2^i) \le 2^{i-1}$, the number of partitions of $i$. 
\begin{problem}
Find the best upper bound for the ${\rank}(\mathcal{J}_2^i)$ as a function of $i$.
\end{problem}
This problem is an exhausted challenge. It seems that the efficient tool is expanding the $Jq^k$ by partitions. The difficulty is that even in lower degrees the partitions are not so well-behaved. For instance, in degree six, the expansion 
\[ 3Jq^3Jq^1Jq^2-4Jq^2Jq^3Jq^1-2Jq^2Jq^1Jq^3+3Jq^1Jq^2Jq^3=0 \]
does not contains the terms $Jq^1Jq^3Jq^2$ and $Jq^3Jq^2Jq^1$.

\vspace{0.2cm}

Studying the dual of dyadic Steenrod is of more importance. The homomorphism $\phi: {\mathcal J}_2 \to {\mathcal A}_2$ induced the homomorphism $\phi^*: {\mathcal J}_2^* \to {\mathcal A}_2^*$ on duals. As well known, the dual ${\mathcal J}_2^*$  is a polynomial algebra \cite{milnor}.

\begin{problem}
What is the nature of the dual ${\mathcal J}_2^*$?
\end{problem}


From Cartan formula $Jq^1(fg)=f Jq^1(g)+Jq^1(f)g $, we have 
\begin{equation}\label{parts}
fg= Jq^{-1}(f Jq^1(g))+ Jq^{-1}(Jq^1(f)g). 
\end{equation}
We call \eqref{parts} the dyadic Steenrod integration by parts. By iterated use of \eqref{parts} we get 
\begin{equation}\label{bypart}
\sum_{k=0}^\infty (-1)^k (Jq^1)^k(f)(Jq^{-1})^{k+1}(g).
\end{equation}
Taking \eqref{bypart} as $F$, we have
\begin{equation*}
Jq^1(F)= -fg+ \lim_{N\to\infty}(-1)^N(Jq^1)^{N+1}(f) (Jq^{-1})^{N+1}(g).
\end{equation*}
Thus for all $f,g\in\mathbb{Z}_2[\xi]$, 
\begin{equation}\label{cartan:inverse}
Jq^{-1}(fg)=\sum_{k=0}^\infty (-1)^k (Jq^1)^k(f)(Jq^{-1})^{k+1}(g)
\end{equation}
is true if and only if 
\[ \lim_{N\to\infty}(-1)^N(Jq^1)^{N+1}(f)(Jq^{-1})^{N+1}(g)=0. \]
Unfortunately, neither $\| \ \|_A$ nor $\|\;\|_L$ makes this limit zero. 
\begin{problem}
Define a new norm on $Q_{cl}^r(\mathcal{J}_2) $ to settle \eqref{cartan:inverse}.
\end{problem}
We may think of \eqref{cartan:inverse} as a Cartan formula for $Jq^{-1}$. 
\begin{problem}
Find Cartan formula for the $Jq^{-k}$.
\end{problem}
Using Hopf algebra methods, it seems that to establish the Cartan formulas for $Jq^{-k}$, there is a need to extend the diagonal map
 $\psi:\mathcal{J}_2\to \mathcal{J}_2\otimes\mathcal{J}_2 $ over $Q_{cl}^r(\mathcal{J}_2)$.
\vspace{0.2cm}


As mentioned in the last part of Section \ref{lint}, in the image of the affinoid space closed ball, under $Jq^1$, is an elliptic curve. 

\begin{problem}
How the properties of closed balls and $Jq^1$ treat in elliptic curves and, conversely, which properties of elliptic curves concern to closed balls? 
\end{problem}
 
As if the derivation of a closed ball is an elliptic curve or, the integration of an elliptic curve is a closed ball, a theory in which the objects of a category can be integrated or differentiated.
\vspace{0.2cm}


The homomorphisms $\psi_q$ introduced in Section \ref{ademn}, have an interesting property. Let $\Psi$ be the group generated
by $\{\psi_q \}_{q\in\mathbb{Q}_2}$. Then we have
\[ \Psi=\{\theta\in\widehat{(\mathcal{J}_2)}_A| \theta\chi(\theta)=\chi(\theta)\theta=1 \} \subsetneqq
\{\theta\in\widehat{(\mathcal{J}_2)}_A| \; \|\theta\|_A=1 \}. \]
Compare $\Psi$ to the subgroup $S^1=\{z\in\mathbb{C}|z\bar{z}=1 \}$ of $\mathbb{C}^\times$. 
Here, any triple $\{1,\psi_q,\chi(\psi_q) \}$ generates a plane, thereafter, a geometry. Note that Euclidean and hyperbolic geometries have a projective completion. For details,
see \cite{green}.

Now, define $\Psi$ to be the unit circle of $\mathbb{Q}_2[\Psi]$. 
\begin{problem}
Construct other conic sections over $\mathbb{Q}_2[\Psi]$. 
\end{problem}
The best way to do this, is utilizing concept of polar and polarity \cite{green}.
\vspace{0.2cm}


The hit problem in Section \ref{ademn} is originated from the ideal ${\mathcal J}_2^+$, in which we defined the Adem norm and then the traditional hit problem posed of the ${\mathcal J}_2^+$-adic topology on ${\mathcal J}_2$. The same case may happens for the symmetric hit problem \cite{criteria,symmetric,generating}.
\begin{problem}
Is there any ideal ${\mathcal I}$ of ${\mathcal J}_2$ for which the ${\mathcal I}$-adic topology on ${\mathcal J_2}$ concerns with the symmetric hit problem?
\end{problem}

\end{document}